%% file: Norm_Minimization_Problem-6+.tex
\documentclass[]{interact}

\usepackage{epstopdf}
\usepackage[caption=false]{subfig}
\usepackage{graphicx}
\usepackage{float}
\usepackage[numbers,sort&compress]{natbib}
\bibpunct[, ]{[}{]}{,}{n}{,}{,}
\makeatletter
\def\NAT@def@citea{\def\@citea{\NAT@separator}}
\makeatother
\usepackage{hyperref}
\hypersetup{
	colorlinks=true,
	linkcolor=blue, 
	citecolor=red, 
	urlcolor=blue  } 
\usepackage{mathptmx}   
\usepackage{enumitem}
\theoremstyle{plain}
\newtheorem{theorem}{Theorem}[section]
\newtheorem{lemma}[theorem]{Lemma}
\newtheorem{corollary}[theorem]{Corollary}
\newtheorem{proposition}[theorem]{Proposition}

\theoremstyle{definition}

\newtheorem{example}[theorem]{Example}

\theoremstyle{remark}
\newtheorem{remark}[theorem]{Remark}

\renewcommand {\theenumi} {\rm\roman{enumi}}

\usepackage{color}
\usepackage{comment}
\usepackage[colorinlistoftodos]{todonotes}
\input{macro}

\usetikzlibrary{patterns,arrows.meta}
\begin{document}

\title{Norm Minimisation Problems Involving Distances to Convex Sets
}

\author{
\name{Doan Huu Hieu\textsuperscript{a,b},
Nguyen Duy Cuong\textsuperscript{c}
}
\thanks{CONTACT Nguyen Duy Cuong. Email: ndcuong@ctu.edu.vn}
\affil{\textsuperscript{a} 
Faculty of Mathematics and Computer Science, University of Science, Ho Chi Minh City, Vietnam\\
\textsuperscript{b}~Vietnam National University, Ho Chi Minh City, Vietnam\\
\textsuperscript{c}~Faculty of Mathematics, College of Natural Sciences, Can Tho University, Can Tho City, Vietnam
}
}
\maketitle

\begin{abstract}
The paper studies product-space norm minimisation problems involving distances to convex sets. 
Using standard tools from convex and functional analysis, we establish complete dual necessary and sufficient optimality conditions and show that the entire solution set can be constructed from the dual vectors arising from the optimality conditions at a given solution.
As a consequence, we study optimality conditions and solution set descriptions for generalised versions of the Fermat-Torricelli problem, the Chebyshev centre problem, and the $p$-Fermat-Torricelli problem. 
Comparisons with existing results are provided whenever applicable. 
Examples in  finite and infinite dimensional  spaces equipped with different norms are presented to illustrate the results.	
\end{abstract}


\begin{amscode}
49J52; 49J53; 49K40; 90C30; 90C46
\end{amscode}

\setcounter{tocdepth}{2}

\section{Introduction}
Given a finite family of nonempty closed convex sets in a normed space, the Generalised Chebyshev Centre Problem (GCCP) \cite{NamHoa13} seeks a ball of minimum radius that intersects every set in the family, whereas the  Generalised Fermat-Torricelli Problem (GFTP) \cite{MorNam11} aims to find a point minimising the sum of the distances to these sets. When all the sets are singletons, these problems reduce to the classical Chebyshev centre problem \cite{AmiMac84} (also known as the Sylvester problem \cite{NamHoaAn14}) and the Fermat-Torricelli problem \cite{MarSwaWei02}, respectively. These models have been extensively studied and have found numerous applications in location theory, data analysis, control theory, approximation theory, and related areas \cite{BolMarSol99,MorNam11,KazLiuMor25,BotTur94}.


The GCCP and GFTP have attracted considerable attention from both computational and theoretical viewpoints. 
In finite-dimensional spaces, numerical algorithms have been proposed for solving these problems for both smooth and nonsmooth norms \cite{NamNguSal12,MorNam11,NamAnRecSun14,BotTur94}. 
On the other hand, studies in infinite-dimensional spaces have mainly focused on dual optimality conditions and qualitative properties of solution sets such as existence, uniqueness, and compactness \cite{NamNguSal12,MorNam11,MorNamVil13,NamHoa13,AliTsa19}.

The GCCP and GFTP have typically been studied separately in the literature. 
This is mainly due to the seemingly different structures of their objective functions: the GCCP minimises the maximum of the distances to the given sets, whereas the GFTP minimises their sum. 
Nevertheless, this distinction is not essential since both problems can be viewed as special cases of a more general norm minimisation problem involving an arbitrary norm on $\R^n$.
To the best of our knowledge, there is no general framework that systematically addresses dual optimality conditions and  solution set representations in which the GCCP and GFTP arise as particular cases. 
Motivated by this observation, the present paper aims to develop such a unified framework for norm minimisation problems involving distances to convex sets.

In many optimisation problems, especially in location theory, it is important to characterise the entire solution set. Such a characterisation not only allows one to select an alternative optimal solution when practical constraints such as regulations or physical restrictions prevent the use of a particular optimal solution but also provides a foundation for sensitivity and stability analysis of the solution set with respect to perturbations of the problem data.
Characterisations of solution sets for convex optimisation problems have been investigated in \cite{Man88,Bur91}.
 In the recent work \cite{Cuo26b},  solution set constructions were established for norm minimisation problems in the special case when all the sets are singletons. 
 In this paper, we make an attempt to extend these results to the more general setting of convex sets.

Employing standard tools from functional and convex analysis, we establish fundamental properties of the solution set such as existence and compactness and derive complete dual necessary and sufficient optimality conditions. 
We then show that the entire solution set admits an explicit representation in terms of the dual vectors associated with the optimality conditions at a given solution.
As a consequence, we specify the established results to the three important norms on $\R^n$: the maximum, sum and $p$-norms. 
This yields  optimality conditions and solution set formulations for the GCCP,  GFTP and  $p$-GFTP.
While some of the optimality conditions for these special cases have appeared in the literature, solution set construction results do not seem to be available.
Several examples in  finite and infinite dimensional spaces are provided to illustrate the results.
In particular, it is shown that the corresponding solution set depends strongly on the choice of norm on the underlying space.

Two norms are used in our model: the norm on the underlying space $X$ and the norm on $\R^n$. 
The norm on $X$ determines the analytical and computational complexity of the problem, whereas the norm on $\R^n$ specifies how the individual distance functions are combined to form the objective function. 
To treat different combinations in a unified manner, we employ the product-norm construction introduced in \cite{BonDun73,SaiKatTak00} and recently refined in \cite{Cuo26a}. 
This construction generates a broad class of norms on $\R^n$ from continuous convex functions satisfying suitable conditions. 
Moreover, the corresponding dual norms admit explicit representations in terms of the generating functions associated with the primal norms providing a convenient framework for our analysis.
In particular, the maximum, sum and $p$-norms are obtained by appropriate choices of the generating function. 

The remainder of the paper is organised as follows. Section~\ref{S2} collects  basic definitions and  results used throughout the paper. 
Section~\ref{S3} investigates a general norm minimisation problem involving distances to convex sets. 
Fundamental properties of the solution set, complete dual necessary and sufficient optimality conditions, and a method for constructing the entire solution set are established. Sections~\ref{S4}-\ref{S6} are devoted to the GFTP,  GCCP and $p$-GFTP, respectively.
Each section includes optimality conditions, solution set descriptions and  illustrative examples.
The final Section~\ref{S7} summarises the main contributions of the paper and outlines potential directions for future research.

\section{Preliminaries}\label{S2}
Our basic notation is standard, see, e.g., \cite{Mor06.1,RocWet98} .
Throughout the paper, if not explicitly stated otherwise, $X$ is a vector space equipped with a norm $\|\cdot\|$.
The topological dual of a vector space $X$ is denoted by $X^*$, while $\langle\cdot,\cdot\rangle$ denotes the bilinear form defining the pairing between the two spaces.
The dual norm of $\|\cdot\|$ is denoted by $\|\cdot\|^*$, and
they are usually labeled by subscripts indicating the underlying space.
The notation $\B_r(\bar x)$ denotes the closed ball in $X$ with centre $\bar x$ and radius $r>0$, while $\B^*$ and $\mathbb{S}^*$ denote the closed unit ball and the unit sphere of the dual space $X^*$, respectively.
The notation $x_\nu^*\xrightarrow{w^*}x^*\in X^*$ denotes the weak$^*$ convergence of the net $\{x_\nu^*\}\subset X^*$.
If $\Omega\subset X^*$, then its weak$^*$ closure is denoted by $\cl^*\Omega$.
When $X$ a Hilbert space, the associated norm is given by  
$\|\cdot\|:=\sqrt{\ang{\cdot,\cdot}}$, and we
identify $X$ and $X^*$ thanks to the Riesz representation theorem \cite[Theorem 5.5]{Bre11}.
The convex hull of a set $\Omega\subset X$ is denoted by $\co\Omega$. 
The distance and projection from a point $\bx\in X$ to $\Omega$ are defined by $d(\bx,\Omega):=\inf_{x\in\Omega}\|\bx-x\|$ and $P_\Omega(\bx):=\{x\in \Omega\mid d(\bx,\Omega)=\|\bx-x\|\}$, respectively.
If $\Omega=\emptyset$, we set $d(\bx,\Omega):=\infty$ and $P_\Omega(\bx):=\emptyset$.
If $\Omega$ is convex, then the function 
$d(\cdot,\Omega)$ is convex.
Symbols $\mathbb{R}$, $\mathbb{R}_+$ and $\mathbb{R}_+^n$ denote the real line equipped with the standard norm, the set of all nonnegative real numbers, and the set of all vectors in $\mathbb{R}^n$ with nonnegative components, respectively.
We write $\infty$ instead of $+\infty$ and employ the convention that $\frac{1}{\infty}=0$ and $\frac{1}{0}=\infty$.
Given a finite family of convex functions
$h_1,\ldots,h_n:X\to\mathbb{R}$, we define the mapping
$h:X\to\mathbb{R}^n$ by
\begin{gather}\label{h}
h(x):=(h_1(x),\ldots,h_n(x))\;\;\text{for all}\;\;x\in X.
\end{gather}


\paragraph*{Subdifferentials and normal cones.}
For a convex function $f:X\to\R\cup\{\infty\}$ on a normed space $X$,
its subdifferential at $\bx\in\dom f:=\{x \in X\mid f(x) <\infty\}$ is given by  
\begin{gather}\label{cs}
	\sd f(\bx)=\left\{x^* \in X^*\mid \ang{x^*,x-\bx}\le f(x)-f(\bx)\;\;\text{for all}\;\;x\in X\right\}.
\end{gather}
By convention, we set $\partial{f}(\bx):=\emptyset$ if $\bx\notin\dom f$.
If $f$ is continuous at $\bx$, then $\partial f(\bx)$ is nonempty and weak$^*$ compact \cite[Theorem 3.25]{Pen13}.
For a norm $\|\cdot\|$ on a vector space $X$,
it is well known \cite[Example~3.36]{MorNam22} that
\begin{equation}\label{sn}
\partial\|\cdot\|(x) = 
\begin{cases}
\left\{x^*\in X^*\mid \|x^*\|^* =1\;\;\text{and}\;\;\langle x^*,x\rangle=\|x\|\right\} & \text{if}\;\;x \ne 0,\\
\{x^*\in X^*\mid \|x^*\|^* \le 1\} & \text{otherwise}.
\end{cases}
\end{equation}	
	


Given a nonempty convex subset $\Omega$ of a normed space $X$, the normal cone to $\Omega$ at $\bx\in\Omega$ is defined by
\begin{gather}\label{CN}
N_\Omega(\bx):=\{x^*\in X^*\mid\langle x^*,x-\bx\rangle\le 0\;\;\text{for all}\;\; x\in\Omega\}.
\end{gather}
If $\bx\notin\Omega$, we set $N_\Omega(\bx):=\emptyset$.
By \cite[Proposition 3.77(a) \& Theorem 3.82]{MorNam22},
\begin{equation}\label{sd}
\partial d(\cdot,\Omega)(\bx)=
\begin{cases}
N_{\Omega}(\bx)\cap \B^* & \text{if}\;\;  \bx\in\Omega,\\
N_{\Omega(\bx)}(\bx)\cap\mathbb{S}^*& \text{if}\;\;\bx\notin\Omega
\end{cases}
\end{equation}
with $\Omega(\bx):=\{x\in X\mid d(x,\Omega)\leq d(\bx,\Omega)\}$.
Note that $\Omega(\bx)$ is nonempty, closed, and convex.

\begin{lemma}{\rm\cite[Proposition 3.70 and Corollary 3.79]{MorNam22}}\label{L2.1}
Let $X$ be a Hilbert space, $\Omega\subset X$ be a nonempty closed  convex set, and $\bx\in X$. 
The following assertions hold.
\begin{enumerate}
\item\label{L2.1-1}
$P_\Omega(\bx)$ is a singleton.
\item\label{L2.1-2}
 If $\bx\notin\Omega$, then $\partial d(\cdot,\Omega)(\bx)$ is also a singleton, and
\begin{gather*}
	\partial d(\cdot,\Omega)(\bx)=
	\left\{\frac{\bx-P_\Omega(\bx)}{d(\bx,\Omega)}\right\}.
\end{gather*} 
\end{enumerate}
\end{lemma}

\begin{lemma}\label{L3.6}
Let $X$ be a normed space,  and $\Omega\subset X$ be a  convex set.
The following assertions hold.
\begin{enumerate}
\item\label{L3.6-1}
Let  $\bx\notin\Omega$ and $P_\Omega(\bx)\ne\emptyset$.
Then 
$\partial d(\cdot,\Omega)(\bx)=\partial\|\cdot-x\|(\bx)\cap N_\Omega(x)$
for all $x\in P_\Omega(\bx).$
\item\label{L3.6-2}
Let $x\in \Omega$, $x^*\in N_\Omega(x)$, and $ z\in X$ with $\langle x^*,z\rangle=\|x^*\|^*\cdot\|z\|$.
Suppose that $x^*\ne 0$ and $z\ne 0$.
Then $x+z\notin\Omega$ and $x\in P_\Omega(x+z)$.
\end{enumerate}	
\end{lemma}	
\begin{proof}
The first assertion can be found in \cite[Proposition~3.77(b)]{MorNam22}.
Suppose that $x+z\in\Omega$.
By \eqref{CN} and the assumption,
\begin{gather*}
0<\|x^*\|^*\cdot\|z\|=\langle x^*,z\rangle= \langle x^*,(x+z)-x\rangle\le 0,
\end{gather*}	
a contradiction.
Thus, $x+z\notin\Omega$.
Let $\omega\in\Omega$.	
By \eqref{CN}, $\langle x^*,x-\omega\rangle\ge 0$.
Then
\begin{align*}
\|x^*\|^*\cdot\|x+z-\omega\|
\ge\langle x^*,x+z-\omega\rangle
=\langle x^*,z\rangle+\langle x^*,x-\omega\rangle
\ge \langle x^*,z\rangle=\|x^*\|^*\cdot\|z\|, 
\end{align*}	 
and consequently, $\|x+z-\omega\|\ge \|z\|$.
Thus, $d(x+z,\Omega)\ge\|z\|$.
On the other hand, $d(x+z,\Omega)\le \|x+z-x\|=\|z\|$.
Thus, $d(x+z,\Omega)=\|z\|=\|(x+z)-x\|$, i.e.,
$x\in P_\Omega(x+z)$.
\end{proof}	

\begin{lemma}\label{L1.2}
Let $X$ be a normed space, $\bar x\in X$, $h_i:X\to\R$ $(i=1,\ldots,n)$ and
$\phi:\R^n\to\R$ be Lipschitz continuous near 
$\bx$ and $h(\bar x)$, respectively, with 
$h$ given by \eqref{h}.
Suppose that $h_1,\ldots,h_n, \phi, \phi\circ h$  are convex, and $\lambda_1,\ldots,\lambda_n\ge 0$ for any $(\lambda_1,\ldots,\lambda_n)\in\partial \phi(h(\bx))$.
Then
\begin{gather}\label{L1.2-1}
\partial(\phi\circ h)(\bx)= \left\{\sum_{i=1}^{n}\lambda_ix^*_i\mid (\lambda_1,\ldots,\lambda_n)\in\partial \phi(h(\bx)),\ x_i^*\in\partial h_i(\bx)\; (i=1,\ldots,n)\right\}.
\end{gather}	
\end{lemma}	

\begin{proof}
Let $\rm R$ denote the set on the right-hand side of \eqref{L1.2-1}.
By \cite[Theorem~2.3.9(i) \& Proposition 2.36(b]{Cla90}, $\partial(\phi\circ h)(\bx)= \cl^*\co{\rm{R}}$.
We now show that $\rm{R}$ is weak$^*$ closed and convex, and {consequently}, $\partial(\phi\circ h)(\bx)= {\rm{R}}$.

By  \cite[Theorem 3.25]{Pen13}, $\partial h_i(\bx)$ $(i=1,\ldots,n)$ are weak$^*$ compact and $\partial\phi(h(\bx))$ is compact.
Thus, the set $K:=\partial \phi(h(\bx))\times\prod_{i=1}^{n}\partial h_i(\bx)$ is compact in $\R^n\times (X^*)^n$ with the product topology.
Define  $\Phi:K\to X^*$ by
\begin{gather*}
\Phi(\lambda,x_1^*,\ldots,x_n^*):=\sum_{i=1}^{n}\lambda_ix_i^*
\end{gather*}	
for all $\lambda:=(\lambda_1,\ldots,\lambda_n)\in\partial \phi(h(\bx)),\ x_i^*\in\partial h_i(\bx)\; (i=1,\ldots,n).$
Let  $\{(\lambda_\nu,{x_{1\nu}^*},\ldots,{x_{n\nu}^*})\}\subset K$ be a net  converging to  some $(\lambda,x_1^*,\ldots,x_n^*)\in K$.
Then $\lambda_\nu\to\lambda$  and ${x_{i\nu}^*}\xrightarrow{w^*} x_i^*$ $(i=1,\ldots,n)$, and consequently,
\begin{equation*}
\langle \Phi(\lambda_\nu,{x_{1\nu}^*},\ldots,{x_{n\nu}^*}),x\rangle
	=\sum_{i=1}^{n}{\lambda_{i\nu}}\langle{x_{i\nu}^*}, x\rangle\to\sum_{i=1}^{n}\lambda_i\langle x_i^*,x\rangle=\langle\Phi(\lambda,x_1^*,\ldots,x_n^*),x\rangle
\end{equation*}
for all $x\in X$.
Hence, $\Phi(\lambda_\nu,{x_{1\nu}^*},\ldots,{x_{n\nu}^*})\xrightarrow{w^*}\Phi(\lambda,x_1^*,\ldots,x_n^*)$.
Thus, $\Phi$ is continuous. 
Observe that $\text{\rm R}=\Phi(K)$.
By \cite[Proposition 1.44]{MorNam22}, 
the compactness of $K$ and the continuity of $\Phi$ implies that
 $\text{\rm R}$ is weak$^*$ compact.
 In particular, it is weak$^*$ closed.

Let $x^*,y^*\in \text{\rm R}$.
Then there exist $\alpha:=(\alpha_1,\ldots,\alpha_n)\in\partial \phi(h(\bx))$, $\beta:=(\beta
_1,\ldots,\beta_n)\in\partial \phi(h(\bx))$, and $x_i^*, y_i^*\in\partial h_i(\bx)$ $(i=1,\ldots,n)$ such that
$x^*:=\sum_{i=1}^{n}\alpha_ix_i^*$ and $y^*:=\sum_{i=1}^{n}\beta_iy_i^*$.
By the assumption, $\al_i,\be_i\ge0$ $(i=1,\ldots,n)$.
Let $\gamma\in(0,1)$ and  $\lambda:=(\lambda_1,\ldots,\lambda_n):=\gamma\alpha+(1-\gamma)\beta$.
In view of the convexity of $\partial \phi(h(\bx))$, we have $\lambda\in\partial \phi(h(\bx))$.
By the assumption, $\lambda_i=\gamma\al_i+(1-\gamma)\be_i\ge 0$ $(i=1,\ldots,n)$.
If $\lambda_i=0$, then $\alpha_i=\beta_i=0$.
For each $i\in\{1,\ldots,n\}$ with $\lambda_i>0$, we have
\begin{gather*}
z^*_i:=\frac{\gamma\alpha_i}{\lambda_i}x_i^*+\frac{(1-\gamma)\beta_i}{\lambda_i}y_i^*\in \partial h_i(\bx)
\end{gather*}	
thanks to the convexity of $ \partial h_i(\bx)$.
Then
\begin{align*}
\gamma x^*+(1-\gamma)y^*
=\sum_{i=1}^{n}\left(\gamma\alpha_ix_i^*+(1-\gamma)\beta_iy_i^*\right)
=\sum_{\lambda_i>0}\left(\gamma\alpha_ix_i^*+(1-\gamma)\beta_iy_i^*\right)
=\sum_{\lambda_i>0}\lambda_i z^*_i\in \text{\rm R}.
\end{align*}
Thus, $\text{\rm R}$ is convex.
This completes the proof.
\end{proof}	

\if{
The following statement is a direct consequence of Lemma~\ref{L1.2} and Proposition~\ref{P1.3}. 
It shows that the operation $\cl^*\co$ in Lemma~\ref{L1.2}\eqref{L1.2-4} is superfluous.
\begin{corollary}\label{C2.4}
Let the assumptions in Lemma~\ref{L1.2} be satisfied,
and ${\rm{R}}$ be given by \eqref{L1.2-2}.
Suppose that
$h_i$ $(i=1,\ldots,n)$ are regular at $\bx$, $\phi$ is regular at $h(\bx)$, and
$\lambda_1,\ldots,\lambda_n\ge 0$ for any $(\lambda_1,\ldots,\lambda_n)\in\partial^C \phi(h(\bx))$. Then $\partial^C(\phi\circ h)(\bx)={\rm{R}}$.
\end{corollary}	
\begin{proof}
By Proposition~\ref{P1.3}, $\rm R$ is weak* compact and convex.
In particular, it is weak* closed.
The statement then follows from Lemma~\ref{L1.2}.	
\end{proof}	
}\fi

\if{
\begin{proof}
Under the assumption made, we have $P_\Omega(z)\ne \emptyset$ for any $z\in X$ \cite[Proposition~3.67]{MorNam22}.
The equivalent is trivially holds when $x^*=0$.
Suppose that $x^*\ne 0$.	
Let $x^*\in \|x^*\|\partial d(\cdot,\Omega)(x)$.
If $x\in\Omega$, then $P_\Omega(x)=\{x\}$.
In this case, the conclusion is obviously satisfied.
Suppose that $x\notin\Omega$ and  let $\omega\in P_\Omega(x)$.
By the definition of the convex subdifferential, 
\begin{gather}\label{L2.6-1}
\langle x^*/\|x^*\|,z-x\rangle\le d(z,\Omega)-d(x,\Omega)\;\;\text{for all}\;\;
z\in X.
\end{gather}	
In view of \eqref{L2.6-1} with $z:=\omega$, we have $\langle x^*/\|x^*\|,\omega-x\rangle\le -d(x,\Omega)=-\|x-\omega\|$, and consequently,
$\langle x^*,x-\omega\rangle\ge \|x^*\|\|x-\omega\|$.
On the other hand, $\langle x^*,x-\omega\rangle\le\|x^*\|\|x-\omega\|$.
Thus, $\langle x^*,x-\omega\rangle=\|x^*\|\|x-\omega\|$.
Let $y\in\Omega$.
In view of \eqref{L2.6-1} with $z:=y$, we have $\langle x^*/\|x^*\|,y-x\rangle\le -\|x-\omega\|$.
Then
\begin{align*}
\langle x^*/\|x^*\|,y-\omega\rangle
=\langle x^*/\|x^*\|,y-x\rangle+\langle x^*/\|x^*\|,x-\omega\rangle
\le -\|x-\omega\|+\|x-\omega\|=0.
\end{align*}	
Thus, $\langle x^*,y-\omega\rangle\le 0$, i.e., $x^*\in N_{\Omega}(\omega)$.
Let $\omega\in P_{\Omega}(x)$, $x^*\in N_{\Omega}(\omega)$ and $\langle x^*,x-\omega\rangle=\|x^*\|\|x-\omega\|$.
Let $y\in X$ and choose a $u\in P_\Omega(y)$.
Then $\langle x^*/\|x^*\|,x-\omega\rangle = d(x,\Omega)$ and $\langle x^*/\|x^*\|,u-y\rangle \ge -\|u-y\| = -d(y,\Omega)$.
Since $x^*\in N_\Omega(\omega)$, we have
$\langle x^*,\omega-u\rangle \ge 0.$
Then
\begin{align*}
\langle x^*/\|x^*\|,x-y\rangle
= \langle x^*/\|x^*\|,x-\omega\rangle
+ \langle x^*/\|x^*\|,\omega-u\rangle
+ \langle x^*/\|x^*\|,u-y\rangle \ge d(x,\Omega) - d(y,\Omega).
\end{align*}
Thus, $x^*\in\partial \|x^*\|d(\cdot,\Omega)(x)$.
This completes the proof.
\end{proof}	
}\fi

\if{
The next statements give a sufficient condition for the condition in Proposition~\ref{P1.3}\;\eqref{P1.3-2}.
\begin{proposition}
Let $\phi:\R^n\to\R$ be convex, and $\bar u\in\R^n$.
Suppose that $\phi(u)\le \phi(v)$ for all $u:=(u_1,\ldots,u_n),v:=(v_1,\ldots,v_n)\in\R^n$ with $u_i\le v_i$ $(i=1,\ldots,n)$.
Then $\lambda_1,\ldots,\lambda_n\ge 0$ for any $(\lambda_1,\ldots,\lambda_n)\in\partial \phi(\bar u)$.
\end{proposition}	
\begin{proof}
Let 
$\lambda:=(\lambda_1,\ldots,\lambda_n)\in\partial \phi(u)$.
Fix $i\in\{1,\ldots,n\}$ and let $\mathbf{e}_i$ be the standard unit vector in $\R^n$.
By the assumption,
$\phi(\bar u-\mathbf{e}_i)\leq \phi(\bar u)$.
By the definition of the convex subdifferential,
\begin{gather*}
\langle \lambda, u-\bar u\rangle\le  \phi(u)-\phi(\bar u)\;\;
\text{for all}\;\;u\in\R^n.
\end{gather*}	
Substituting $u:=\bar u-\mathbf{e}_i$, we obtain $\langle\lambda,-\mathbf{e}_i\rangle\le \phi(\bar u-\mathbf{e}_i)-\phi(\bar u)\le 0$.
Thus,
$\langle\lambda,\mathbf{e}_i\rangle\le 0$, and consequently, $\lambda_i\ge 0$.
\end{proof}	
}\fi

\paragraph*{A norm construction on $\R^n$.}
A norm 
\begin{gather}\label{phi}
\phi:=\|\cdot\|_{\R^n}
\end{gather}	
 is called 
\textit{sign-symmetric} \cite{Cuo26a} if $\phi(x_1,\ldots,x_n)=\phi(\pm x_1,\ldots,\pm x_n)$ for all
$(x_1,\ldots,x_n)\in \R^n$.
We now recall a construction of sign-symmetric norms on $\R^n$ via a class of convex functions; see \cite{BonDun73,SaiKatTak00,Cuo26a}.

Let $n\ge 2$ and 
$\Delta_n:=\{(t_1,\ldots,t_{n})\in\R^{n}_+\mid \sum_{i=1}^nt_i =1\}$.
Denote by $\pmb{\Psi}_n$ the family of all convex continuous  functions $\psi:\Delta_n\to\R$ such that
$\psi(\mathbf{e}_1)=\cdots=\psi(\mathbf{e}_n)=1$, where $\mathbf{e}_1,\ldots,\mathbf{e}_n$ are the standard basis vectors of $\R^n$, and
\begin{gather*}
\psi(t_1,\ldots,t_n)\ge (1-t_i)\cdot\psi\left(\dfrac{t_1}{1-t_i},\ldots,\dfrac{t_{i-1}}{1-t_i},0,\dfrac{t_{i+1}}{1-t_i},\ldots,\dfrac{t_{n}}{1-t_i}\right)
\end{gather*}	
for all $(t_1,\ldots,t_n)\in\Delta_n$ with $t_i<1$ 
$(i=1,\ldots,n)$.
A function $\psi\in \pmb{\Psi}_n$ is called \textit{permutation-symmetric} if $\psi(t_1,\ldots,t_n)=\psi(t_{\sigma(1)},\ldots,t_{\sigma(n)})$ for all $(t_1,\ldots,t_n)\in\Delta_n$ and every permutation
$\sigma$ of $\{1,\ldots,n\}$.

The next statement collects some basic facts needed for the subsequent analysis; see \cite[Lemma~1, Proposition~2, Theorems~6 \& 10]{Cuo26a}.
\begin{proposition}\label{P2.2}
Let $\psi\in\pmb{\Psi}_n$.
The following assertions hold.
\begin{enumerate}
\item\label{P2.2-1}
$\max\{t_1,\ldots,t_n\}\le\psi(t)\le 1$ for all $t:=(t_1,\ldots,t_n)\in\Delta_n$.
\item\label{P2.2-3}
If $0<\alpha_i\le\beta_i$ $(i=1,\ldots,n)$, then 
\begin{equation*}
	\left(\sum_{i=1}^{n}\alpha_i\right)\psi\left(\frac{\alpha_1}{\sum_{i=1}^{n}\alpha_i},\ldots,\frac{\alpha_n}{\sum_{i=1}^{n}\alpha_i}\right)
	\le\left(\sum_{i=1}^{n}\be_i\right)\psi\left(\frac{\be_1}{\sum_{i=1}^{n}\be_i},\ldots,\frac{\be_n}{\sum_{i=1}^{n}\be_i}\right).
\end{equation*}
\item\label{P2.2-2}
If $\psi$ is \textit{permutation-symmetric}, then $\psi$ attains its minimium at $\left(\frac{1}{n},\ldots,\frac{1}{n}\right)$.
\item\label{P2.2-5} 
The function
\begin{equation}\label{T2.2-1}
\|x\|_\psi
:= \begin{cases}
\left(\sum_{i=1}^n|x_i|\right)\cdot\psi\left(\dfrac{|x_1|}{\sum_{i=1}^n|x_i|},\ldots,\dfrac{|x_{n}|}{\sum_{i=1}^n|x_i|}\right)  & \text{\rm if } x\ne0,\\
0 & \text{\rm otherwise}
\end{cases} 
\end{equation}
for all $x:=(x_1,\ldots,x_n)\in \R^n$ is a norm on $\R^n$.
\item\label{P2.2-6}
The  dual norm of \eqref{T2.2-1} is given by
\begin{gather*}
\|x\|^*_{\psi} =\max_{(t_1,\ldots,t_n)\in\Delta_n}\dfrac{\sum_{i=1}^nt_i|x_i|}{\psi (t_1,\ldots,t_{n})}
\end{gather*}	
for all $x:=(x_1,\ldots,x_n)\in\R^n$.
\end{enumerate}
\end{proposition}	

\begin{remark}\label{R2.5}
\begin{enumerate}
\item\label{R2.5-0}
Let $\psi\in\pmb{\Psi}_n$.
By \eqref{T2.2-1},  $\|\cdot\|_\psi$  is sign-symmetric.
By Proposition~\ref{P2.2}\eqref{P2.2-3}, 
$\|x\|_\psi\le \|y\|_\psi$ for all $x:=(x_1,\ldots,x_n), y:=(y_1,\ldots,y_n)\in\R^n$ with $|x_i|\leq |y_i|$ $(i=1,\ldots,n)$.
Define 
\begin{gather}\label{psi1}
\psi^*(s):=\max_{t\in\Delta_n}\frac{\langle t,s\rangle}{\psi(t)}\;\;\text{for all}\;\;s\in\Delta_n.
\end{gather}	
By Proposition~\ref{P2.2}\eqref{P2.2-6},   $\|\cdot\|^*_\psi=\|\cdot\|_{\psi^*}$.
 Moreover, one can verify that $\psi^*\in\pmb{\Psi}_n$ \cite[Proposition~4(i)]{Cuo26a}.
\item 
Let $p\in[1,\infty]$.
The function
\begin{gather*}
\psi_p(t_1,\ldots,t_n)
:=\begin{cases}
\left(t_1^p+\cdots+t_n^p\right)^{\frac{1}{p}}  & \text{if}\;\;p\in[1,\infty),\\
\max\{t_1,\ldots,t_n\} & \text{if}\;\;p=\infty
\end{cases} 
\end{gather*}
for all $(t_1,\ldots,t_n)\in\Delta_n$ belongs to $\pmb{\Psi}_n$ and is  permutation-symmetric.
By \eqref{T2.2-1},
\begin{equation*}
\|x\|_{\psi_p}
=\begin{cases}
\left(|x_1|^p+\cdots+|x_n|^p\right)^{\frac{1}{p}}  & \text{if}\;\;p\in[1,\infty),\\	
\max\{|x_1|,\ldots,|x_n|\} & \text{if}\;\; p=\infty.
\end{cases} 
\end{equation*}
By \eqref{psi1}, we can compute that
$\psi^*_p=\psi_q$ with $q\in[1,\infty]$ satisfying $\frac{1}{p}+\frac{1}{q}=1$; see \cite{Cuo26a}.
The reader is referred to \cite{SaiKatTak00} for further examples of functions in  $\pmb{\Psi}_n$.
\item\label{R2.5-1}
By Proposition~\ref{P2.2}\eqref{P2.2-1},
$\|\cdot\|_{\psi_\infty}\le \|\cdot\|_\psi\le \|\cdot\|_{\psi_1}$.
\end{enumerate}
\end{remark}

\section{A norm minimisation problem involving distances to convex sets}\label{S3}
Let $(X,\|\cdot\|)$ be a normed space, and $\Omega_1,\ldots,\Omega_n\subset X$ be nonempty closed convex sets.
Define $h_i(x):=d(x,\Omega_i)$ for all $x\in X$ $(i=1,\ldots,n)$.
Let $h$ and $\phi$ be given by \eqref{h} and \eqref{phi}, respectively.
Note $h_1,\ldots,h_n,\phi$ are convex and globally Lipschitz continuous.
Consider the unconstrained minimisation problem:
\begin{gather}
\label{P}
\tag{$P$}
{\rm{minimise}}\; f(x):=(\phi\circ h)(x)\;\;\text{subject to}\;\;x\in X.
\end{gather}
Denote by $\text{\rm Sol}(P,\|\cdot\|_{\R^n})$ the set of all solutions to problem \eqref{P}.

\begin{proposition}\label{P3.1}
Let $\phi$ be given by \eqref{phi}.	
Suppose that the {following} conditions are satisfied:
\begin{enumerate}[label=(A\arabic*)]
\item\label{A1} 
$\phi$ is sign-symmetric;
\item\label{A2} 
$\phi(\al)\le\phi(\be)$ for all $\al:=(\al_1,\ldots,\al_n),\be:=(\be_1,\ldots,\be_n)\in\R^n$ with
	$|\al_i|\le |\be_i|$ $(i=1,\ldots,n)$.
\end{enumerate}
Then $f$ is convex and globally Lipschitz continuous.
\if{
The following assertions hold.
\begin{enumerate}
\item\label{P3.1-1}
The function 
\item\label{P3.1-2}
Let $(X,\|\cdot\|)$ is a reflexive Banach space, and
$f(x)\to\infty$ as $\|x\|\to\infty$.
Then $\text{\rm {Sol}}(P,\|\cdot\|_{\R^n})$ is nonempty closed and convex. 
\end{enumerate}
}\fi
\end{proposition}	

\begin{proof}
For any $x,y\in X$, we have
$|h_i(x)-h_i(y)|\le \|x-y\|$ $(i=1,\ldots,n)$, and consequently,
\begin{align*}
\abs{f(x)-f(y)}
&=\abs{\phi(h(x))-\phi(h(y))}\\
&\le \phi(h_1(x)-h_1(y),\ldots,h_n(x)-h_n(y))\\
&= \phi(|h_1(x)-h_1(y)|,\ldots,|h_n(x)-h_n(y)|)\\
&\le \phi(\|x-y\|,\ldots,\|x-y\|)\\
&=\phi(1,\ldots,1)\cdot\|x-y\|.
\end{align*}	
Thus, $f$ is globally Lipschitz continuous.	
	
For any $x,y\in X$ and $\lambda\in[0,1]$, we have
\begin{align*}
f(\lambda x+(1-\lambda)y)
&=\phi(h_1(\lambda x+(1-\lambda)y),\ldots,h_n(\lambda x+(1-\lambda)y))\\
&\le\phi(\lambda h_1(x)+(1-\lambda)h_1(y),\ldots,\lambda h_n(x)+(1-\lambda)h_n(y))\\
&=\phi(\lambda h(x)+(1-\lambda)h(y))\\
&\le \lambda\phi (h(x))
+(1-\lambda)\phi(h(y))\\
&=\lambda f(x)+(1-\lambda)f(y).
\end{align*}	 	
Thus, $f$ is convex.
\end{proof}	

\begin{remark}\label{R3.1}
By Remark~\ref{R2.5}\eqref{R2.5-0}, the norm $\phi:=\|\cdot\|_\psi$ given by \eqref{T2.2-1} for some $\psi\in\pmb{\Psi}_n$  satisfies \ref{A1} and \ref{A2}.
\end{remark}	

\begin{proposition}\label{P4.3}
Let  $(X,\|\cdot\|)$ be a reflexive Banach space, $\phi:=\|\cdot\|_\psi$ for some $\psi\in \pmb{\Psi}_n$, and one of the sets $\Omega_1,\ldots,\Omega_n$ be bounded.
The following assertions hold:
\begin{enumerate}
\item\label{P4.3-1}
$\text{\rm {Sol}}(P,\|\cdot\|_\psi)$  is nonempty, closed, convex and bounded;
\item\label{P4.3-2}
if \rm{dim}\;$X<\infty$, then $\text{\rm {Sol}}(P,\|\cdot\|_\psi)$  is compact.
\end{enumerate}
\end{proposition}	

\begin{proof}
We first prove \eqref{P4.3-1}.
Without loss of generality, we can assume that $\Omega_1$ is bounded.
Then there exists a $\gamma>0$ such that $\|x\|<\gamma$ for all $x\in\Omega_1$.
By Remark~\ref{R2.5}\eqref{R2.5-1},
\begin{align*}
f(x) 
\ge \max_{1\le i\le n}h_i(x)
\ge h_1(x)
=\inf_{u\in\Omega_1}\|x-u\|\ge \inf_{u\in\Omega_1}(\|x\|-\|u\|)\ge \|x\|-\gamma
\end{align*}	
for all $x\in X$.
Thus,
\begin{gather}\label{P4.3-3}
f(x)\to\infty\;\;\text{as}\;\;\|x\|\to\infty.
\end{gather}	
By \cite[Corollary~3.23]{Bre11}, $\text{\rm {Sol}}(P,\|\cdot\|_\psi)$ is nonempty. 
By Proposition~\ref{P3.1} and Remark~\ref{R3.1}, $\text{\rm {Sol}}(P,\|\cdot\|_\psi)$ is closed and convex.
Suppose that there exists a sequence $\{x_k\}\subset \text{\rm {Sol}}(P,\|\cdot\|_\psi)$ such that $\|x_k\|\to\infty$ as $k\to\infty$.
By \eqref{P4.3-3}, we have $f(x_k)\to\infty$ as $k\to\infty$, a contradiction.
Thus, $\text{\rm {Sol}}(P,\|\cdot\|_\psi)$ is bounded.
Assertion \eqref{P4.3-2} is a direct consequence of \eqref{P4.3-1}.
\end{proof}	

\begin{remark}
When $\phi:=\|\cdot\|_{\psi_\infty}$, the nonemptiness of the solution set in assertion \eqref{P4.3-1} was established in \cite[Proposition~3.2]{NamNguSal12}.
When $\Omega_1,\ldots,\Omega_n$ are singletons, assertion \eqref{P4.3-2} recaptures \cite[Proposition~\ref{P4.3}(iv)]{Cuo26b}.
\end{remark}

\if{
The following result provides an explicit formula for the subdifferential of a composition involving convex functions. Since we have been unable to locate an appropriate reference for this result in the general normed-space setting, we include a complete proof for the reader's convenience. A finite dimensional version can be found in \cite[Theorem 4.3.1]{HirLem04}.
\begin{theorem}\label{T3.4}
Let $X$ be a normed space, $\bx\in X$, $h_i:X\to\R$ $(i=1,\ldots,n)$ be convex and locally Lipschitz continuous at 
$\bx$,  $h:=(h_1,\ldots,h_n)$.
Suppose that $\phi:\R^n\to\R$ is convex and locally Lipschitz at $h(\bx)$ 
satisfying
$\lambda_1,\ldots,\lambda_n\ge 0$ for any $(\lambda_1,\ldots,\lambda_n)\in\partial \phi(h(\bx))$.
Then
\begin{equation}\label{L3.4-1}
\partial(\phi\circ h)(\bx)=\left\{\sum_{i=1}^{n}\lambda_ix^*_i\mid (\lambda_1,\ldots,\lambda_n)\in\partial \phi(h(\bx)),\ x_i^*\in\partial h_i(\bx)\; (i=1,\ldots,n)\right\}.
\end{equation}
\end{theorem}
}\fi

\if{
\begin{proof}
Let $\lambda=(\lambda_1,\ldots,\lambda_n)\in\partial g(h(\bar{x})),\ x_i^*\in\partial h_i(\bar{x})$ $(i=1,\ldots,n)$, and define $x^*:=\sum_{i=1}^{n}\lambda_ix^*_i$. Since $ x_i^*\in\partial h_i(\bar{x})$ $(i=1,\ldots,n)$,  $h_i(x)-h_i(\bar{x})\ge \langle x_i^*,x-\bar{x}\rangle$ for all $x\in X$. Then
\begin{equation}\label{P2.4-1}
\sum_{i=1}^{n}\lambda_i(h_i(x)-h_i(\bar{x}))\ge \left\langle\sum_{i=1}^{n}\lambda_ix_i^*,x-\bar{x}\right\rangle=\langle x^*,x-\bar{x}\rangle\;\;\text{for all}\;\; x\in X.
\end{equation}
Since $\lambda\in\partial g(h(\bar{x}))$, $g(y)-g(h(\bar{x}))\ge \langle\lambda,y-h(\bar{x})\rangle$ for all $y\in\R^n$. Then
\begin{equation}\label{P2.4-2}
g(h(x))-g(h(\bar{x}))\ge\sum_{i=1}^{n}\lambda_i(h_i(x)-h_i(\bar{x}))\;\;\text{for all}\;\; x\in X.
\end{equation}
By \eqref{P2.4-1} and \eqref{P2.4-2},
\begin{equation*}
f_{\vertiii{\cdot}}(x)-f_{\vertiii{\cdot}}(\bar{x})\ge \langle x^*,x-\bar{x}\rangle\;\;\text{for all}\;\; x\in X.
\end{equation*}
Thus, $x^*\in \partial f_{\vertiii{\cdot}}(\bar{x})$. 
\end{proof}
}\fi
The following result provides a formula for the subdifferential of the objective function of problem~\eqref{P}.
\begin{proposition}\label{P4.5}
Suppose that conditions \ref{A1} and \ref{A2} are satisfied, and $\bar{x}\notin\bigcup_{i=1}^n\Omega_i$.
\if{
Then
\begin{multline}\label{P4.5-0}
\partial f(\bar x)\subset\cl^*\co\Bigg\{\sum_{i=1}^{n}\lambda_ix^*_i\mid
\|(\lambda_1,\ldots,\lambda_n)\|_{\R^n}^*\le1,\;x_i^*\in \partial d(\cdot,\Omega_i)(\bar{x})\;(i=1,\ldots,n),\\
\sum_{i=1}^{n}\lambda_id(\bar{x},\Omega_i)=\|(d(\bar{x},\Omega_1),\ldots,d(\bar{x},\Omega_n))\|_{\R^n}
\Bigg\}.
\end{multline}
}\fi
Then
\begin{multline}\label{P4.5-1}
\partial f(\bar x)=\Bigg\{\sum_{i=1}^{n}\lambda_ix^*_i\mid
\lambda_i\ge 0,\;x_i^*\in \partial d(\cdot,\Omega_i)(\bx)\;(i=1,\ldots,n),\;\|(\lambda_1,\ldots,\lambda_n)\|_{\R^n}^*=1,\\\;\;\;\;
\sum_{i=1}^{n}\lambda_id(\bar{x},\Omega_i)=\|(d(\bar{x},\Omega_1),\ldots,d(\bar{x},\Omega_n))\|_{\R^n}
\Bigg\}.
\end{multline}
\end{proposition}	

\begin{proof}
Recall that $f=\phi\circ h$ with  $h$ and $\phi$ given by \eqref{h} and \eqref{phi}, respectively.
By Proposition~\ref{P3.1}, $f$ is convex and Lipschitz continuous.
Let $\lambda:=(\lambda_1,\ldots,\lambda_n)\in\partial \phi(h(\bx))$. 
By \eqref{sn}, the latter is equivalent to $\|(\lambda_1,\ldots,\lambda_n)\|_{\R^n}^*=1$ and 
$\sum_{i=1}^{n}\lambda_id(\bar{x},\Omega_i)=\|(d(\bar{x},\Omega_1),\ldots,d(\bar{x},\Omega_n))\|_{\R^n}$.
We now show that $\lambda_1,\ldots,\lambda_n\ge 0$.
Fix $i\in\{1,\ldots,n\}$ and  let $\mathbf{e}_i$ be the standard basis vector in $\R^n$. 
By \ref{A2}, 
\begin{gather}\label{P4.5-3}
\phi(h(\bx)-h_i(\bx)\mathbf{e}_i)\leq \phi(h(\bx)).
\end{gather}	
By \eqref{cs},
$\langle \lambda,u-h(\bx)\rangle\le  \phi(u)-\phi(h(\bx))$ for all
$u\in\R^n$.
Substituting $u:=h(\bx)-h_i(\bx)\mathbf{e}_i$, we obtain
\begin{gather}\label{P4.5-4}
\langle\lambda,-h_i(\bx)\mathbf{e}_i\rangle\le \phi(h(\bx)-h_i(\bx)\mathbf{e}_i)-\phi(h(\bx)).
\end{gather}	
By \eqref{P4.5-3} and \eqref{P4.5-4},
$\langle\lambda,-h_i(\bx)\mathbf{e}_i\rangle\le 0$, and consequently, $\lambda_ih_i(\bx)\ge 0$.
 Note that $h_i(\bar{x})>0$ since $\bar{x}\notin\Omega_i$ and $\Omega_i$ is closed.
 Hence, $\lambda_i\ge 0$.
By Lemma~\ref{L1.2}, we obtain \eqref{P4.5-1}.
\end{proof}	

The next result establishes dual necessary and sufficient optimality conditions for problem~\eqref{P}.
\begin{theorem}\label{T3.5}
Suppose that conditions \ref{A1} and \ref{A2} are satisfied, and $\bar{x}\notin\bigcup_{i=1}^n\Omega_i$.
The following assertions hold.
\begin{enumerate}
\item\label{T3.5-1}
$\bar x\in \text{\rm Sol}(P,\|\cdot\|_{\R^n})$ if and only if there exist $(\lambda_1,\ldots,\lambda_n)\in\R^n_+$ and $x_i^*\in \partial d(\cdot,\Omega_i)(\bx)$ $(i=1,\ldots,n)$ such that
\begin{gather}\label{P2.5-0}
\sum_{i=1}^{n}\lambda_ix^*_i=0, \|(\lambda_1,\ldots,\lambda_n)\|_{\R^n}^*=1,
\sum_{i=1}^{n}\lambda_id(\bar{x},\Omega_i)=\|(d(\bar{x},\Omega_1),\ldots,d(\bar{x},\Omega_n))\|_{\R^n}.
\end{gather}	
\item\label{T3.5-2}
Let $(\lambda_1,\ldots,\lambda_n)\in\R^n_+$ and $x_i^*\in \partial d(\cdot,\Omega_i)(\bx)$ $(i=1,\ldots,n)$ satisfy \eqref{P2.5-0}.
Then
\begin{multline}\label{P2.5-1}
\text{\rm Sol}(P,\|\cdot\|_{\R^n})=\Bigg\{x\in X\mid
x_i^*\in \partial d(\cdot,\Omega_i)(x)\;\;\text{with}\;\;\lambda_i>0,
\\\;\;\;\;
\sum_{\lambda_i>0}\lambda_id(x,\Omega_i)=\|(d(x,\Omega_1),\ldots,d(x,\Omega_n))\|_{\R^n}
\Bigg\}.
\end{multline}
\end{enumerate}
\end{theorem}	

\begin{proof}
By the standard fact of convex analysis, $\bar x\in \text{\rm Sol}(P,\|\cdot\|_{\R^n})$ if and only if $0\in \partial f(\bar x)$.
Thus, characterisation \eqref{P2.5-0} is a consequence of 
\eqref{P4.5-1}.
This proves \eqref{T3.5-1}.

We now prove \eqref{T3.5-2}.
Denote by $\text{\rm N}$ the set in the right-hand side of \eqref{P2.5-1}.
Let $x\in\text{\rm N}$.
Then
\begin{gather*}
x_i^*\in \partial d(\cdot,\Omega_i)(x)\;\;\text{with}\;\;\lambda_i>0\;\;\text{and}\;\;
\sum_{\lambda_i>0}\lambda_id(x,\Omega_i)=\|(d(x,\Omega_1),\ldots,d(x,\Omega_n))\|_{\R^n}.
\end{gather*}	
Since $x_i^*\in\partial d(\cdot,\Omega_i)(x)\cap \partial d(\cdot,\Omega_i)(\bx)$, it follows from \eqref{cs} that
$d(x,\Omega_i)-d(\bx,\Omega_i)=\langle x_i^*,x-\bx\rangle$.
From this and the first equality in \eqref{P2.5-0},
\begin{equation*}
\sum_{\lambda_i>0}\lambda_id(x,\Omega_i)-\sum_{\lambda_i>0}\lambda_id(\bx,\Omega_i)=\left\langle\sum_{\lambda_i>0}\lambda_ix_i^*,x-\bx\right\rangle=0.
\end{equation*}
Thus, $\sum_{i=1}^{n}\lambda_id(x,\Omega_i)=\sum_{i=1}^{n}\lambda_id(\bx,\Omega_i)$, and consequently,
\begin{gather*}
\|(d(x,\Omega_1),\ldots,d(x,\Omega_n))\|_{\R^n}=\|(d(\bar{x},\Omega_1),\ldots,d(\bar{x},\Omega_n))\|_{\R^n}.
\end{gather*}	
The latter is the optimal value of \eqref{P}.
Thus, $x\in \text{\rm Sol}(P,\|\cdot\|_{\R^n})$.
  
Let $x\in\text{\rm Sol}(P,\|\cdot\|_{\R^n})$. 
Since $x_i^*\in \partial d(\cdot,\Omega_i)(\bx)$, we have
\begin{equation}\label{T3.7-1}
d(x,\Omega_i)-d(\bx,\Omega_i)\ge\langle x_i^*,x-\bx\rangle\;\;(i=1,\ldots,n).
\end{equation}
Then
\begin{equation*}
	\sum_{i=1}^{n}\lambda_id(x,\Omega_i)-\sum_{i=1}^{n}\lambda_id(\bx,\Omega_i)\ge \sum_{i=1}^{n}\langle\lambda_ix_i^*,x-\bx\rangle=\left\langle\sum_{i=1}^{n}\lambda_ix_i^*,x-\bx\right\rangle=0.
\end{equation*}
Thus,
\begin{align*}
\sum_{i=1}^{n}\lambda_id(\bx,\Omega_i)
&\le	\sum_{i=1}^{n}\lambda_id(x,\Omega_i)\\
&\le\|(\lambda_1,\ldots,\lambda_n)\|^*_{\R^n}\cdot\|(d(x,\Omega_1),\ldots,d(x,\Omega_n))\|_{\R^n}\\
&\le\|(d(x,\Omega_1),\ldots,d(x,\Omega_n))\|_{\R^n}\\
&=\|(d(\bar{x},\Omega_1),\ldots,d(\bar{x},\Omega_n))\|_{\R^n}\\
&=\sum_{i=1}^{n}\lambda_id(\bar{x},\Omega_i),
\end{align*}	
and consequently,
\begin{gather*}
\sum_{\lambda_i>0}\lambda_id(x,\Omega_i)=\|(d(x,\Omega_1),\ldots,d(x,\Omega_n))\|_{\R^n}\;\;\text{and}\;\; \sum_{\lambda_i>0}\lambda_id(x,\Omega_i)=\sum_{\lambda_i>0}\lambda_id(\bx,\Omega_i).
\end{gather*}	
By the latter and the first equality in \eqref{P2.5-0}, 
\begin{gather*}
\sum_{\lambda_i>0}\lambda_i\left(d(x,\Omega_i)-d(\bx,\Omega_i)-\langle x_i^*,x-\bx\rangle\right)=\sum_{\lambda_i>0}\lambda_id(x,\Omega_i)-\sum_{\lambda_i>0}\lambda_id(\bx,\Omega_i)-\sum_{\lambda_i>0}\langle\lambda_ix_i^*,x-\bx\rangle =0.
\end{gather*}
From this and \eqref{T3.7-1}, $d(x,\Omega_i)-d(\bx,\Omega_i)=\langle x_i^*,x-\bx\rangle$ for all $\lambda_i>0$. 
Fix $i\in\{1,\ldots,n\}$ with $\lambda_i>0$.
Let $u\in X$. 
Using the fact that $x_i^*\in\partial d(\cdot,\Omega_i)(\bx)$, we obtain
\begin{align*}
\langle x_i^*,u-x\rangle
&=\langle x_i^*,u-\bx\rangle+\langle x_i^*,\bx-x\rangle\\
&\le (d(u,\Omega_i)-d(\bx,\Omega_i))+(d(\bx,\Omega_i)-d(x,\Omega_i))
=d(u,\Omega_i)-d(x,\Omega_i).
\end{align*}
Thus, $x_i^*\in \partial d(\cdot,\Omega_i)(\bx)$.
Hence $x\in \text{\rm N}$.
This completes the proof.
\end{proof}	

\begin{remark}
In view of \eqref{P2.5-1}, a complete description of the solution set can be obtained if we know the dual vectors
$(\lambda_1,\ldots,\lambda_n)\in\mathbb{R}^n_+$ and
$x^*_1,\ldots,x^*_n$ in \eqref{P2.5-0}  corresponding to a given solution
$\bar x\notin\bigcup_{i=1}^n\Omega_i$.
\end{remark}

The following example illustrates  Theorem~\ref{T3.5} for the case $n=2$.
\begin{example}\label{E3.7}
Let $x_1,x_2\in X$, and $r>0$ with $\|x_1-x_2\|>2r$.
Define 
$\Omega_1:=\B_{r}(x_1)$, $\Omega_2:=\B_{r}(x_2)$, and  $\bar x:=\frac{x_1+x_2}{2}$.
Observe that $\Omega_1, \Omega_2$ are closed and convex, and $\bar{x}\notin\Omega_1\cup\Omega_2$.
Suppose that $\psi\in \pmb{\Psi}_2$ is permutation-symmetric, and $\phi:=\|\cdot\|_\psi$ is given by \eqref{T2.2-1}.
By Remark~\ref{R3.1}, $\phi$ satisfies conditions \ref{A1} and \ref{A2}.
Then there exist $\lambda:=\psi(\frac{1}{2},\frac{1}{2})$ and $x_1^*\in \partial d(\cdot,\Omega_1)(\bx)$, $x_2^*\in \partial d(\cdot,\Omega_2)(\bx)$ such that
\begin{gather}
 x^*_1+ x_2^*=0,\; \|(\lambda,\lambda)\|^*_{\psi}=1,\label{E3.7-4}\\
\lambda d(\bar{x},\Omega_1)+\lambda d(\bar{x},\Omega_2)=\|(d(\bar{x},\Omega_1),d(\bar{x},\Omega_2))\|_\psi.\label{E3.7-5}
\end{gather}	
By Theorem~\ref{T3.5}, $\bar x\in \text{\rm Sol}(P,\|\cdot\|_\psi)$ and
\begin{multline*}
\text{\rm Sol}(P,\|\cdot\|_\psi)=\Big\{x\in X\mid x_1^*\in  \partial d(\cdot,\Omega_1)(x),\; x_2^*\in  \partial d(\cdot,\Omega_2)(x),\\ \lambda d(x,\Omega_1)+\lambda d(x,\Omega_2)=\|(d(x,\Omega_1),d(x,\Omega_2))\|_\psi\Big\}.
\end{multline*}
\end{example}	

\begin{proof}
The objective function is given by $f(x):=\|(d(x,\Omega_1),d(x,\Omega_2))\|_\psi$ for all $x\in X$.
We first prove that 
\begin{gather}\label{E3.7-6}
d(x,\Omega_i)=\max\{\|x-x_i\|-r,0\}\;\;\text{for all}\;\;x\in X\;\;(i=1,2).
\end{gather}	
Let $i\in\{1,2\}$ and $x\in X$.
If $x\in\Omega_i$, then \eqref{E3.7-6} is obvious. 
Suppose that $x\notin\Omega_i$.
Then $\|x-x_i\|-r>0$. 
Define $z_i:=x_i+\frac{r}{\|x-x_i\|}(x-x_i)$. 
Then $\|z_i-x_i\|=r$, i.e., $z_i\in\Omega_i$, and consequently, 
\begin{gather*}
d(x,\Omega_i)\le \|x-z_i\|=\|x-x_i\|-r=\max\{\|x-x_i\|-r,0\}.
\end{gather*}	
On the other hand, for any $u\in\Omega_i$, we have $\|u-x_i\|\le r$, and consequently,
$\|x-u\|\geq \|x-x_i\|-\|u-x_i\|\geq \|x-x_i\|-r.$
Thus, $d(x,\Omega_i)=\inf_{u\in\Omega_i}\|x-u\|\ge\|x-x_i\|-r$. 
Therefore, condition \eqref{E3.7-6} is satisfied.
By the calculus rule for the maximum function \cite[Theorem 3.59]{MorNam22} and the fact
$d(\bx,\Omega_i)=\|\bx-x_i\|-r>0$, we have
$\partial d(\cdot,\Omega_i)(\bx)=\partial\|\cdot-x_i\|(\bx)$.
Note that $\bx\ne x_i$  $(i=1,2)$.
By \eqref{sn},
\begin{gather}\label{E3.7-7}
\partial d(\cdot,\Omega_i)(\bx)=\left\{x_i^*\in X^*\mid\|x_i^*\|^*=1,\;\langle x_i^*,\bx-x_i\rangle=\|\bx-x_i\|\right\}\; (i=1,2).
\end{gather}	

We now show that there exist $x_1^*\in \partial d(\cdot,\Omega_1)(\bx)$ and $x_2^*\in \partial d(\cdot,\Omega_2)(\bx)$ such that conditions \eqref{E3.7-4} and \eqref{E3.7-5} are satisfied.
Since $\bar{x}-x_1\ne 0$, 
the Hahn-Banach theorem \cite[Corollary~1.3]{Bre11} guarantees the existence of $x^*_1\in X^*$ such that $\|x^*_1\|^*=1$ and $\langle x^*_1,\bx-x_1\rangle=\|\bx-x_1\|$.
By \eqref{E3.7-7}, $x_1^*\in \partial d(\cdot,\Omega_1)(\bx)$.
Let $x_2^*:=-x_1^*$. 
Then $x^*_1+x^*_2=0$, $\|x^*_2\|^*=1$ and $$\langle x^*_2,\bx-x_2\rangle=\langle x^*_1,\bx-x_1\rangle=
\|\bx-x_1\|=\|\bx-x_2\|.$$
By \eqref{E3.7-7}, $x_2^*\in \partial d(\cdot,\Omega_2)(\bx)$.
Let $\lambda:=\psi(\frac{1}{2},\frac{1}{2})$. 
By Proposition~\ref{P2.2}\eqref{P2.2-2}, $\ds\min_{(t_1,t_2)\in\Delta_2}\psi(t_1,t_2)=\lambda$. 
By Proposition~\ref{P2.2}\eqref{P2.2-6},
\begin{equation*}
\|(\lambda,\lambda)\|^*_{\psi}=\max_{(t_1,t_2)\in\Delta_2}\frac{t_1\cdot\lambda+t_2\cdot\lambda}{\psi(t_1,t_2)}=\frac{\lambda}{\ds\min_{(t_1,t_2)\in\Delta_2}\psi(t_1,t_2)}=1.
\end{equation*}
In view of \eqref{E3.7-6}, we have $d(\bx,\Omega_1)=d(\bx,\Omega_2)$. 
From this and  \eqref{T2.2-1}, we obtain \eqref{E3.7-5}.
The conclusion is a direct consequence of Theorem~\ref{T3.5}.
\end{proof}	

\begin{remark}
When $X=\mathbb{R}^m$ is equipped with the Euclidean norm, the characterisation  \eqref{E3.7-6} and 
the  subdifferential  formula \eqref{E3.7-7} were studied in \cite[Proposition~3.1]{NamHoaAn14}.
\end{remark}	

\begin{proposition}\label{C3.8}
Let $X$ be a Hilbert space, and $\psi\in \pmb{\Psi}_n$.
Suppose that one of the sets $\Omega_1,\ldots,\Omega_n$ is bounded.
Then 
\begin{gather*}
\emptyset\ne \text{\rm Sol}(P,\|\cdot\|_\psi)\subset \co\left(\bigcup_{i=1}^{n}{\Omega}_i\right).
\end{gather*}	
\end{proposition}	

\begin{proof}
Every Hilbert space is a  reflexive Banach space. 
By Proposition~\ref{P4.3}, $\text{\rm Sol}(P,\|\cdot\|_\psi)$ is nonempty.
Let $x\in \text{\rm Sol}(P,\|\cdot\|_\psi)$.
If ${x}\in\bigcup_{i=1}^{n}{\Omega}_i$, then ${x}\in\co\left(\bigcup_{i=1}^{n}{\Omega}_i\right)$.
Suppose that ${x}\notin\bigcup_{i=1}^{n}{\Omega}_i$. 
By Theorem~\ref{T3.5}\eqref{T3.5-1}, there exist $(\lambda_1,\ldots,\lambda_n)\in\R^n_+$ and $x_i^*\in \partial d(\cdot,\Omega_i)(x)$ $(i=1,\ldots,n)$ such that condition \eqref{P2.5-0} is satisfied with $x$ in place of $\bx$.
By Lemma~\ref{L2.1}\eqref{L2.1-1}, there exists a unique $u_i\in{\Omega}_i$ such that $d(x,\Omega_i)=\|x-u_i\|$ $(i=1,\ldots,n)$.
By Lemma~\ref{L2.1}\eqref{L2.1-2},
$x_i^*=\frac{x-u_i}{d(x,\Omega_i)}$ $(i=1,\ldots,n)$. 
Let $\mu_i:=\frac{\lambda_i}{d(x,\Omega_i)}\ge 0$ $(i=1,\ldots,n)$ and $\mu:=\sum_{i=1}^{n}\mu_i>0$. 
Then 
\begin{gather*}
0=\sum_{i=1}^{n}\lambda_ix^*_i
=\sum_{i=1}^{n}\mu_i(x-u_i)
=\mu x-\sum_{i=1}^{n}\mu_iu_i,
\end{gather*}	
and consequently,
\begin{equation*}
x=\frac{\mu_1}{\mu}u_1+\cdots+\frac{\mu_n}{\mu}u_n\in\text{\rm co}\left(\bigcup_{i=1}^{n}{\Omega}_i\right).
\end{equation*}
The proof is complete.
\end{proof}	

\begin{remark}
When $\Omega_1,\ldots,\Omega_n$ are singletons, Proposition~\ref{C3.8} recaptures \cite[Proposition~4.9]{Cuo26b}.
\end{remark}

\if{
The next example illustrates the result in Corollary~\ref{C3.8}.
\begin{example}
Let $X:=\R^2$ be equipped with the Euclidean norm $\|\cdot\|$, and $v_1,v_2,v_3$
be three distinct points. 
Consider the classical Fermat–Weber problem, i.e.,
\begin{gather*}
{\rm{minimise }}\;\;f(u):=\|v_1-u\|+\|v_2-u\|+\|v_3-u\|\;\;\text{for all}\;\;
u\in \R^2.
\end{gather*}
If none of the angles of the triangle formed by $v_1,v_2,v_3$ is greater than or equal to $120^\circ$, the unique solution lies inside the triangle. 
Otherwise, if one angle is $120^\circ$ or larger, the solution coincides with the vertex at that angle.
In either case, the solution is unique and lies within the convex hull of the three points.
\end{example}
}\fi

\section{Generalised Fermat-Torricelli Problem}\label{S4}
We study a particular case of problem \eqref{P} with $\phi:=\|\cdot\|_{\psi_1}$.
The objective function is of the form:
\begin{gather}\label{4.1}
f(x)=d(x,\Omega_1)+\cdots+d(x,\Omega_n)\;\;\text{for all}\;\;x\in X.
\end{gather}	

The following statement provides optimality conditions and a formula for the solution set of the GFTP.
\begin{theorem}\label{T2.7}
Let $\bar{x}\notin\bigcup_{i=1}^n\Omega_i$.
The following assertions hold.
\begin{enumerate}
\item\label{T2.7.1}
$\bar x\in\text{\rm Sol}(P,\|\cdot\|_{\psi_1})$ if and only if there exist $x_1^*,\ldots,x_n^*\in X^*$ such that
\begin{equation}\label{T2.7-1}
x_i^*\in \partial d(\cdot,\Omega_i)(\bx)\;\;(i=1,\ldots,n)\;\;\text{and}\;\; \sum_{i=1}^{n}x^*_i=0.
\end{equation}
\item\label{T2.7.2}
Let $x_1^*,\ldots,x^*_n\in X^*$  satisfy \eqref{T2.7-1}.
Then $\text{\rm Sol}(P,\|\cdot\|_{\psi_1})=\bigcap_{i=1}^n \text{\rm \textbf{A}}(\Omega_i,x_i^*)$
with
\begin{gather}\label{A+}
\text{\rm\textbf{A}}(\Omega_i,x^*_i):=\{x\in X\mid x^*_i\in\partial d(\cdot,\Omega_i)(x)\}.
\end{gather}
\end{enumerate}
\end{theorem}

\begin{proof}
By Theorem~\ref{T3.5}\eqref{T3.5-1}, $\bar x\in \text{\rm Sol}(P,\|\cdot\|_{\psi_1})$ if and only if	 there exist $(\lambda_1,\ldots,\lambda_n)\in\R^n_+$ and $x_i^*\in  \partial d(\cdot,\Omega_i)(\bx)$ $(i=1,\ldots,n)$ such that
\begin{gather}\label{T3.5-10}
\sum_{i=1}^{n}\lambda_ix^*_i=0,\; \max_{1\le i\le n}\lambda_i=1,\;
\sum_{i=1}^{n}\lambda_id(\bar{x},\Omega_i)=\sum_{i=1}^{n}d(\bar{x},\Omega_i).
\end{gather}	
Since $d(\bx,\Omega_i)>0$ $(i=1,\ldots,n)$ and $ \max_{1\le i\le n}\lambda_i=1$, the last equality in \eqref{T3.5-10} is equivalent to  $\lambda_1=\ldots=\lambda_n=1$.
This proves \eqref{T2.7.1}.
Assertion \eqref{T2.7.2} is a direct consequence of Theorem~\ref{T3.5}\eqref{T3.5-2}.
\end{proof}		

\begin{remark}
\begin{enumerate}
\item	
A nonconvex counterpart of Theorem~\ref{T2.7}\eqref{T2.7.1} in terms of limiting normal cones and subdifferentials in Asplund spaces can be found in \cite[Theorems~4.1 \& 4.3]{MorNam11}.
\item 
When $\Omega_1,\ldots,\Omega_n$ are singletons, Theorem~\ref{T2.7} recaptures \cite[Theorem~5.1]{Cuo26b}.
\end{enumerate}
\end{remark}	

The following statement provides a characterisation of the set \eqref{A+}.
\begin{proposition}\label{P3.7}
Let $\Omega\subset X$ be a convex set, and $x^*\in\mathbb{S}^*$.
Suppose that $P_\Omega(x)\ne \emptyset$ for any $x\in X$.
Then
$
\text{\rm\textbf{A}}(\Omega,x^*)=\text{\rm H}(\Omega,x^*)+\text{\rm T}(x^*)
$
with
\begin{gather}\label{P3.7-1}
\text{\rm H}(\Omega,x^*):=\{x\in \Omega\mid x^*\in N_\Omega(x)\}\;\;\text{and}\;\;
\text{\rm T}(x^*):=\{z\in X\mid\langle x^*,z\rangle=\|z\|\}.
\end{gather}
\end{proposition}	

\begin{proof}
Let $\bx\in \text{\rm\textbf{A}}(\Omega,x^*)$.
If $\bx\in\Omega$, then by \eqref{sd}, $\bx\in \text{\rm H}(\Omega,x^*)$, and consequently,
$\bx=\bx+0\in \text{\rm H}(\Omega,x^*)+\text{\rm T}(x^*)$.
Suppose  that $\bx\notin\Omega$.
By Lemma~\ref{L3.6}\eqref{L3.6-1}, 
$x^* \in N_\Omega(x)$ and
$\langle x^*, \bx-x \rangle = \|\bx-x\|$ for all $x \in P_\Omega(\bx)$.
Define $z:=\bx-x$ for some  $x \in P_\Omega(\bx)$. 
Then $z\in\text{\rm T}(x^*)$, and consequently, $\bx = x + z \in \text{\rm H}(\Omega,x^*)+\text{\rm T}(x^*)$.
	
Conversely, let 
$\bx := x + z$ for some $x \in \text{\rm H}(\Omega,x^*)$ and $z \in \text{\rm T}(x^*)$.
If $z=0$, then $x^*\in \partial d(\cdot,\Omega)(\bx)$.
By \eqref{sd} and the fact that $\|x^*\|^*=1$, we have $\bx\in \text{\rm\textbf{A}}(\Omega,x^*)$.
Suppose that $z\ne 0$.
Then $\langle x^*,\bx-x\rangle=\|\bx-x\|$.
By Lemma~\ref{L3.6}\eqref{L3.6-2}, we have $\bx\notin \Omega$ and $x\in P_\Omega(\bx)$.
By Lemma~\ref{L3.6}\eqref{L3.6-1}, $x^*\in\partial d(\cdot,\Omega)(\bx)$.
Thus, $\bx\in \text{\rm\textbf{A}}(\Omega,x^*)$.
\end{proof}

\begin{remark}
By \cite[Proposition~3.67]{MorNam22},  condition $P_\Omega(x)\ne \emptyset$ for any $x\in X$ is implied by any
of the following conditions:
\begin{enumerate}
\item 
$\Omega$ is compact;
\item 
$\Omega$ is closed and $\dim X<\infty$;
\item 
$\Omega$ is closed and convex, and $X$ is a reflexive Banach space. 
\end{enumerate}	
\end{remark}	

\if{\begin{remark}
\begin{enumerate}
\item
It follows from \eqref{T3.9-1} that $x^*_i\in\partial\|\cdot\|(v_i-\bar u)$ if $\bar u\ne v_i$.  
\item 
A finite-dimensional Minkowski space version of Theorem~\ref{T2.7} can be found in \cite[Theorems~3.1 and 3.2]{MarSwaWei02}.
When $X:=\R^n$ is equipped with the Euclidean norm, Theorem~\ref{T2.7} improves \cite[Reformulation~18.4\;(i)]{BolMarSol99}.
\item 
A result similar to Theorem~\ref{T2.7}, in the setting of a sum of convex functions in finite-dimensional spaces, can be found in \cite[Lemma~3.1]{DurMic85}.
\end{enumerate}
\end{remark}	
}\fi

The following examples illustrate Theorem~\ref{T2.7}.
\begin{example}\label{E4.4}
Let $X:=\R^2$ be equipped with the norm $\|\cdot\|:=\|\cdot\|_{\psi_p}$ with $p\in[1,\infty]$, and
\begin{gather} \label{E4.4-1}
x_1:=(-2,0),\;x_2:=(2,0),\;\bar x:=(0,0),\\ \label{E4.4.2}
\Omega_1:=\B_1(x_1),\;\;\Omega_2:=\B_1(x_2).
\end{gather}	
The dual norm is $\|\cdot\|^*=\|\cdot\|_{\psi_q}$ with $q\in[1,\infty]$ satisfying $\frac{1}{p}+\frac{1}{q}=1$.
Observe that $\Omega_1, \Omega_2$ are closed and convex, and $\bar{x}\notin\Omega_1\cup\Omega_2$.
The objective function \eqref{4.1} is of the form
\begin{gather}\label{E4.4-0}
f(x)=d(x,\Omega_1)+d(x,\Omega_2)
\end{gather}	
for all $x\in\R^2.$
By Example~\ref{E3.7}, we have $\bar x\in\text{\rm Sol}(P)$.
Define $x_1^*:=(1,0)$ and $x_2^*:=(-1,0)$.
Then $x_1^*+x_2^*=0$, and
\begin{gather*}
\|x_1^*\|^*=\|x_2^*\|^*=1,\;\ang{x_1^*,\bx-x_1}=\|\bx-x_1\|=2,\;\ang{x_2^*,\bx-x_2}=\|\bx-x_2\|=2.
\end{gather*}
By \eqref{E3.7-7}, $x_1^*\in\partial d(
\cdot,\Omega_1)(\bx)$ and $x_2^*\in\partial d(
\cdot,\Omega_2)(\bx)$.
\begin{enumerate}
\item\label{E4.4+1}
Suppose that $\|\cdot\|:=\|\cdot\|_{\psi_\infty}$.
By \eqref{P3.7-1},
\begin{align}
\text{\rm H}(\Omega_1,x_1^*)&=\{(-1,u)\mid-1\le u\le1\},\;
\text{\rm T}(x_1^*)=\{(u_1,u_2)\mid u_1\ge|u_2|\},\label{E4.4.3}\\
\text{\rm H}(\Omega_2,x_2^*)&=\{(1,u)\mid-1\le u\le1\},\;
\text{\rm T}(x_2^*)=\{(u_1,u_2)\mid u_1\le-|u_2|\}.\label{E4.4.5}
\end{align}	
\begin{figure}[H]
\begin{tikzpicture}[scale=0.9]
\begin{scope}
\draw[step=1cm,gray!50,dashed] (-3.5,-2.5) grid (3.5,2.5);

\foreach \x in {-2,0,2}{
\draw (\x,0.08)--(\x,-0.08) node[anchor=north,xshift=4pt] at (\x,0) {\small $\x$};}
\foreach \y in {-2,-1,1,2}
\draw (0.08,\y)--(-0.08,\y) node[left] {\small $\y$};
\draw[-{Stealth}](-3.5,0)--(3.5,0) node[below right] {$x$};
\draw[-{Stealth}](0,-2.5)--(0,2.5) node[above left] {$y$};
			
\fill[blue, fill opacity=0.5] (-3,1) rectangle (-1,-1);
\node[blue, below] at (-2,1.6) {$\Omega_1$};
\draw[blue, thick] 	
(-3,1) rectangle (-1,-1);
			
\fill[red, fill opacity=0.5] (1,1) rectangle (3,-1);
\node[red, below] at (2,1.6) {$\Omega_2$};
\draw[red, thick] (1,1) rectangle (3,-1);
			
\draw[blue, line width=2pt] (-1,1.02)--(-1,-1.02);
\node[blue, below] at (-1,-1) {$\text{\rm H}(\Omega_1,x_1^*)$};
			
\draw[red, line width=2pt] 	
(1,1.02)--(1,-1.02);
\node[red, below] at (1,-1) {$\text{\rm H}(\Omega_2,x_2^*)$};
			
\fill(-2,0) circle (2pt);
\node[above] at (-2,0) {$x_1$};
			
\fill(2,0) circle (2pt);
\node[above] at (2,0) {$x_2$};
\end{scope}
		
\begin{scope}[shift={(8,0)}]
\draw[step=1cm,gray!50,dashed] (-3.5,-2.5) grid (3.5,2.5);

\foreach \x in {-3,-2,-1,1,2,3}{
\draw (\x,0.08)--(\x,-0.08) node[anchor=north,xshift=4pt] at (\x,0) {\small $\x$};}
\foreach \y in {-2,-1,1,2}
\draw (0.08,\y)--(-0.08,\y) node[left] {\small $\y$};
\node[below] at (0.15,-0.1) {$0$};	
\draw[-{Stealth}](-3.5,0)--(3.5,0) node[below right] {$x$};
\draw[-{Stealth}](0,-2.5)--(0,2.5) node[above left] {$y$};

\clip (-3.5,-2.5) rectangle (3.5,2.5);
	
\fill[red, fill opacity=0.5]
(-3.5,3.5)--(0,0)--(-3.5,-3.5);
\draw[red, thick] (-3.5,3.5)--(0,0)--(-3.5,-3.5);
\node[red] at (-2.5,1.5) {$\text{\rm T}(x_2^*)$};

\fill[blue, fill opacity=0.5]
(3.5,3.5)--(0,0)--(3.5,-3.5);
\draw[blue, thick] (3.5,3.5)--(0,0)--(3.5,-3.5);
\node[blue] at (2.5,1.5) {$\text{\rm T}(x_1^*)$};

\fill(0,0) circle (1pt);
\end{scope}
\end{tikzpicture}
\caption{Example \ref{E4.4}\eqref{E4.4+1}: Sets \eqref{E4.4.2}, \eqref{E4.4.3} and \eqref{E4.4.5}
}
\label{fig1}
\end{figure}
\noindent
By Proposition~\ref{P3.7},
\begin{align}
\text{\rm \textbf{A}}(\Omega_1,x_1^*)
&=\text{\rm H}(\Omega_1,x_1^*)+
\text{\rm T}(x_1^*)
=\left\{(u_1,u_2)\mid u_1\ge-1,\; |u_2|\le 2+u_1
\right\},\label{E4.4.4}\\
\text{\rm \textbf{A}}(\Omega_2,x_2^*)
&=\text{\rm H}(\Omega_2,x_2^*)+
\text{\rm T}(x_2^*)
=\left\{(u_1,u_2)\mid u_1\le1,\; |u_2|\le2-u_1
\right\}.\label{E4.4.6}
\end{align}	
By Theorem~\ref{T2.7}\eqref{T2.7.2}, 
\begin{align}\notag
\text{\rm Sol}(P)&=\text{\rm \textbf{A}}(\Omega_1,x_1^*)\cap \text{\rm \textbf{A}}(\Omega_2,x_2^*)\\
&=\left\{(u_1,u_2)\mid 
|u_1|\le1,\;|u_2|\le\min\{2+u_1,2-u_1\}
\right\}. \label{E4.4.6+}
\end{align}

\begin{figure}[H]
\begin{tikzpicture}[scale=0.9]
\begin{scope}
\draw[step=1cm,gray!50,dashed] (-3.5,-2.5) grid (3.5,2.5);

\foreach \x in {-3,-2,-1,0,1,2,3}{
\draw (\x,0.08)--(\x,-0.08) node[anchor=north,xshift=4pt] at (\x,0) {\small $\x$};}
\foreach \y in {-2,-1,1,2}
\draw (0.08,\y)--(-0.08,\y) node[left] {\small $\y$};
\draw[-{Stealth}](-3.5,0)--(3.5,0) node[below right] {$x$};
\draw[-{Stealth}](0,-2.5)--(0,2.5) node[above left] {$y$};

\clip (-3.5,-2.5) rectangle (3.5,2.5);
	
\fill[red, fill opacity=0.5]
(-3.5,3)--(-1,3)--(1,1)--(1,-1)--(-1,-3)--(-3.5,-3);
\draw[red, thick] (-3.5,3)--(-1,3)--(1,1)--(1,-1)--(-1,-3)--(-3.5,-3);
\node[red] at (-2,1.5) {$\text{\rm \textbf{A}}(\Omega_2,x_2^*)$};
	
\fill[blue, fill opacity=0.5]
(3.5,3)--(1,3)--(-1,1)--(-1,-1)--(1,-3)--(3.5,-3);
\draw[blue, thick] (3.5,3)--(1,3)--(-1,1)--(-1,-1)--(1,-3)--(3.5,-3);
\node[blue] at (2,1.5) {$\text{\rm \textbf{A}}(\Omega_1,x_1^*)$};
\end{scope}

\begin{scope}[shift={(8,0)}]
\draw[step=1cm,gray!50,dashed] (-2.5,-2.5) grid (2.5,2.5);

\foreach \x in {-1,0,1}{
\draw (\x,0.08)--(\x,-0.08) node[anchor=north,xshift=4pt] at (\x,0) {\small $\x$};}
\foreach \y in {-2,2}
\draw (0.08,\y)--(-0.08,\y) node[left] {\small $\y$};
\draw[-{Stealth}](-2.5,0)--(2.5,0) node[below right] {$x$};
\draw[-{Stealth}](0,-2.5)--(0,2.5) node[above left] {$y$};

\fill[red, fill opacity=0.6]
(0,2)--(-1,1)--(-1,-1)--(0,-2)--(1,-1)--(1,1)--(0,2);
\draw[red, thick] (0,2)--(-1,1)--(-1,-1)--(0,-2)--(1,-1)--(1,1)--(0,2);

\node[blue] at (1.5,1.5) {Sol$(P)$};
\end{scope}
\end{tikzpicture}
\caption{Example \ref{E4.4}\eqref{E4.4+1}:
Sets \eqref{E4.4.4}, \eqref{E4.4.6} and \eqref{E4.4.6+} }
\label{fig1-2}
\end{figure}

\if{
\begin{figure}[H]
\begin{tikzpicture}[scale=0.8]
\draw[step=1.5cm,gray!50,dashed] (-2.5,-3.5) grid (2.5,3.5);

\fill[red, fill opacity=0.8]
(0,3)--(-1.5,1.5)--(-1.5,0)--(-1.5,-1.5)--(0,-3)--(1.5,-1.5)--(1.5,1.5)--(0,3);
\draw[red, thick] (0,3)--(-1.5,1.5)--(-1.5,0)--(-1.5,-1.5)--(0,-3)--(1.5,-1.5)--(1.5,1.5)--(0,3);
\node at (0.2,-0.3){$0$};
\node at (1.7,-0.3){$1$};
\node at (-1.5,-0.3){$-1$};
\foreach \y in {-2,2}
\node[left] at (-0.08,1.5*\y) {$\y$};
\draw[->] (-2.5,0)--(2.5,0) node[below right] {$x$};
\draw[->] (0,-3.5)--(0,3.5) node[above left] {$y$};
\end{tikzpicture}
{\scriptsize
Example~\ref{E4.4}\eqref{E4.4-1} \hspace{4cm} 
Example~\ref{E4.4}\eqref{E4.4-3}
}
\caption{Solution sets in Example~\ref{E4.4}}
\label{fig2}
\end{figure}
}\fi

\item\label{E4.4-3}
Suppose that $\|\cdot\|:=\|\cdot\|_{\psi_p}$ with $p\in[1,\infty)$.
By \eqref{P3.7-1},
\begin{align}
\text{\rm H}(\Omega_1,x_1^*)&=\{(-1,0)\},\;
\text{\rm T}(x_1^*)=\{(u,0)\mid u\ge0\},\label{E4.4.7}\\
\text{\rm H}(\Omega_2,x_2^*)&=\{(1,0)\},\;
\text{\rm T}(x_2^*)=\{(u,0)\mid u\le0\}.\label{E4.4.9}
\end{align}	

\begin{figure}[H]
\begin{tikzpicture}[scale=0.9]
\begin{scope}
\draw[step=1cm,gray!50,dashed] (-3.5,-2.5) grid (3.5,2.5);

\draw[-{Stealth}](-3.5,0)--(3.5,0) node[below right] {$x$};
\draw[-{Stealth}](0,-2.5)--(0,2.5) node[above left] {$y$};
\foreach \x in {-2,0,2}{
\draw (\x,0.08)--(\x,-0.08) node[anchor=north,xshift=4pt] at (\x,0) {\small $\x$};}
\foreach \y in {-2,-1,1,2}
\draw (0.08,\y)--(-0.08,\y) node[left] {\small $\y$};

\fill(-2,0) circle (2pt);
\node[above] at (-2,0) {$x_1$};

\fill(2,0) circle (2pt);
\node[above] at (2,0) {$x_2$};

\node[blue, below] at (-2,1.6) {$\Omega_1$};


\fill[blue, fill opacity=0.5]
plot[domain=-3:-1,samples=200]
(\x,{(1-abs(\x+2)^3)^(1/3)})
--
plot[domain=-1:-3,samples=200]
(\x,{-(1-abs(\x+2)^3)^(1/3)})
--cycle;

\draw[blue, thick]
plot[domain=-3:-1,samples=200]
(\x,{(1-abs(\x+2)^3)^(1/3)})
--
plot[domain=-1:-3,samples=200]
(\x,{-(1-abs(\x+2)^3)^(1/3)})
--cycle;

\node[red, below] at (2,1.6) {$\Omega_2$};
			

\fill[red, fill opacity=0.5]
plot[domain=1:3,samples=200]
(\x,{(1-abs(\x-2)^3)^(1/3)})
--
plot[domain=3:1,samples=200]
(\x,{-(1-abs(\x-2)^3)^(1/3)})
--cycle;

\draw[red, thick]
plot[domain=1:3,samples=200]
(\x,{(1-abs(\x-2)^3)^(1/3)})
--
plot[domain=3:1,samples=200]
(\x,{-(1-abs(\x-2)^3)^(1/3)})
--cycle;
\end{scope}

\begin{scope}[shift={(8,0)}]
\draw[step=1cm,gray!50,dashed] (-3.5,-2.5) grid (3.5,2.5);

\foreach \x in {-1,0,1}{
\draw (\x,0.08)--(\x,-0.08) node[anchor=north,xshift=4pt] at (\x,0) {\small $\x$};}
\foreach \y in {-2,-1,1,2}
\draw (0.08,\y)--(-0.08,\y) node[left] {\small $\y$};
\draw[-{Stealth}](-3.5,0)--(3.5,0) node[below right] {$x$};
\draw[-{Stealth}](0,-2.5)--(0,2.5) node[above left] {$y$};

\draw[-{Stealth}, blue, very thick] (0,0)--(3.5,0);
\node[blue] at (2.2,-0.5) {$\text{\rm T}(x_1^*)$};

\draw[red, very thick] (-3.5,0)--(0,0);
\node[red] at (-2.2,-0.5) {$\text{\rm T}(x_2^*)$};

\fill[blue](-1,0) circle (3pt);
\node[blue] at (-1,0.5) {$\text{\rm H}(\Omega_1,x_1^*)$};

\fill[red](1,0) circle (3pt);
\node[red] at (1,0.5) {$\text{\rm H}(\Omega_2,x_2^*)$};
\end{scope}
\end{tikzpicture}
\caption{Example \ref{E4.4}\eqref{E4.4-3}: Sets \eqref{E4.4.2},
\eqref{E4.4.7} and \eqref{E4.4.9}}
\end{figure}
\noindent
By Proposition~\ref{P3.7},
\begin{align}
\text{\rm \textbf{A}}(\Omega_1,x_1^*)
&=\text{\rm H}(\Omega_1,x_1^*)+
\text{\rm T}(x_1^*)
=\left\{(u,0)\mid u\ge-1
\right\},\label{E4.4.8}\\
\text{\rm \textbf{A}}(\Omega_2,x_2^*)
&=\text{\rm H}(\Omega_2,x_2^*)+
\text{\rm T}(x_2^*)
=\left\{(u,0)\mid u\le1
\right\}.\label{E4.4.10}
\end{align}	
By Theorem~\ref{T2.7}\eqref{T2.7.2}, 
\begin{align}\label{E4.4.11}
\text{\rm Sol}(P)=\text{\rm \textbf{A}}(\Omega_1,x_1^*)\cap \text{\rm \textbf{A}}(\Omega_2,x_2^*)
=\left\{(u,0)\mid 
-1\le u\le 1
\right\}.
\end{align}

\begin{figure}[H]
\begin{tikzpicture}[scale=0.9]
\begin{scope}
\draw[step=1cm,gray!50,dashed] (-3.5,-2.5) grid (3.5,2.5);
	
\foreach \x in {-2,-1,0,1,2}{
\draw (\x,0.08)--(\x,-0.08) node[anchor=north,xshift=4pt] at (\x,0) {\small $\x$};}
\foreach \y in {-2,-1,1,2}
\draw (0.08,\y)--(-0.08,\y) node[left] {\small $\y$};
\draw[-{Stealth}] (-3.5,0)--(3.5,0) node[below right] {$x$};
\draw[-{Stealth}] (0,-2.5)--(0,2.5) node[above left] {$y$};

\draw[-{Stealth}, blue, very thick] (-1,0)--(3.5,0);
\draw[red, very thick] (-3.5,0)--(1,0);
\draw[violet, very thick] (-1,0)--(1,0);

\node[red] at (-1.5,0.5) {$\text{\rm \textbf{A}}(\Omega_2,x_2^*)$};
	
\node[blue] at (1.5,0.5) {$\text{\rm \textbf{A}}(\Omega_1,x_1^*)$};
\end{scope}

\begin{scope}[shift={(8,0)}]
\draw[step=1cm,gray!50,dashed] (-2.5,-2.5) grid (2.5,2.5);
	
\draw[-{Stealth}] (-2.5,0)--(2.5,0) node[below right] {$x$};
\draw[-{Stealth}] (0,-2.5)--(0,2.5) node[above left] {$y$};
\foreach \x in {-2,-1,0,1,2}{
\draw (\x,0.08)--(\x,-0.08) node[anchor=north,xshift=4pt] at (\x,0) {\small $\x$};}
\foreach \y in {-2,-1,1,2}
\draw (0.08,\y)--(-0.08,\y) node[left] {\small $\y$};
	
\fill[red, very thick]
(-1,0)--(1,0);
\draw[red, very thick] (-1,0)--(1,0);
	
\fill[red](-1,0) circle (2pt);
\fill[red](1,0) circle (2pt);

\node[blue] at (0.7,0.5) {Sol$(P)$};
\end{scope}
\end{tikzpicture}
\caption{Example \ref{E4.4}\eqref{E4.4-3}: Sets
\eqref{E4.4.8}, \eqref{E4.4.10} and \eqref{E4.4.11}}
\end{figure}
\end{enumerate}
\end{example}

\begin{example}\label{E4.3}
Let $X:=L^2([0,1])$.
It is a Hilbert space  with the inner product	
\begin{gather*}
\ang{x,y}:=\displaystyle\int_{0}^1 x(t)y(t)dt\;\;\text{for all}\;\;
x,y\in X.
\end{gather*}	
Define
\begin{gather}\label{E5.4-1}
x_1(t):=t\;\;\text{for all}\;\;t\in[0,1],\;\;x_2:=-x_1,\;\;\bar x:=0,\\
\Omega_1:=\B_{1/2}(x_1),\;\;
\Omega_2:=\B_{1/2}(x_2)
\label{E5.4-2}.
\end{gather}	
Then $\Omega_1, \Omega_2$ closed and convex, and $\bar{x}\notin\Omega_1\cup\Omega_2$.
The objective function is of the form \eqref{E4.4-0} for all $x\in X$.
By Example~\ref{E3.7}, $\bar x\in\text{\rm Sol}(P)$.
Define $x_1^*:=-\sqrt{3}x_1$ and $x_2^*:=-\sqrt{3}x_2$.
Then $x_1^*+x_2^*=0$ and
\begin{gather*}
\|x_1^*\|^*=\|x_2^*\|^*=1,\;\ang{x_1^*,\bx-x_1}=\|\bx-x_1\|={1}/{\sqrt3},\;\ang{x_2^*,\bx-x_2}=\|\bx-x_2\|={1}/{\sqrt3}.
\end{gather*}
By \eqref{E3.7-7}, $x_1^*\in\partial d(
\cdot,\Omega_1)(\bx)$ and $x_2^*\in\partial d(
\cdot,\Omega_2)(\bx)$.
Fix $i\in\{1,2\}$. By \eqref{P3.7-1},
\begin{gather}\notag
\text{\rm T}(x_i^*)=\{z\in X\mid\langle x_i^*,z\rangle=\|z\|\}
=\{\gamma x_i\mid\gamma\le 0\},
\\
\text{\rm H}(\Omega_i, x_i^*)=\{\omega\in\Omega_i\mid x_i^*\in N_{\Omega_i}(\omega)\}
=\{\omega\in\Omega_i\mid\langle x_i, u_i-\omega\rangle\ge 0\;\;\text{for all}\;\; u\in\Omega_i\}
\label{E5.4-6}.
\end{gather}
Consider the  function $u\mapsto\langle x_i,u\rangle$ on $\Omega_i$ and let $\omega_i:=\left(1-\frac{\sqrt{3}}{2}\right)x_i$.
Observe that $\omega_i\in\Omega_i$ since $\|\omega_i-x_i\|=\frac{1}{2}$.
On the other hand, 
\begin{align}\label{E5.4-9}
	\langle x_i,u\rangle=\langle x_i,x_i\rangle+\langle x_i,u-x_i\rangle\ge\frac{1}{3}-\|x_i\|\cdot\|u-x_i\|\ge\frac{1}{3}-\frac{1}{2\sqrt{3}}
\end{align}
for all $u\in \Omega_i$, and
$\langle x_i,\omega_i\rangle=\frac{1}{3}-\frac{1}{2\sqrt{3}}$. Thus, $\omega_i$ is a minimiser.
Let $u_i\in\Omega_i$ be any minimiser.
By \eqref{E5.4-9}, 
$\|u_i-x_i\|=\frac{1}{2}$ and $u_i-x_i=\lambda x_i$ for some $\lambda\le 0$.
Thus, $|\lambda|\cdot\|x_i\|=\frac{1}{2}$, and consequently, $\lambda=-\frac{\sqrt{3}}{2}$.
Then $u_i=\omega_i$.
By \eqref{E5.4-6}, any element of $\text{\rm H}(\Omega_i, x_i^*)$ is a minimiser of the function $\langle x_i,\cdot\rangle$ on $\Omega_i$, and consequently,
$\text{\rm H}(\Omega_i, x_i^*)=\left\{\left(1-\frac{\sqrt{3}}{2}\right)x_i\right\}$.
By Proposition~\ref{P3.7},
\begin{align}
\text{\rm \textbf{A}}(\Omega_1,x_1^*)
&=\text{\rm H}(\Omega_1,x_1^*)+
\text{\rm T}(x_1^*)
=\left\{\gamma x_1\mid \gamma\le 1-\sqrt 3/2\right\},\label{E5.4-7}\\
\text{\rm \textbf{A}}(\Omega_2,x_2^*)
&=\text{\rm H}(\Omega_2,x_2^*)+
\text{\rm T}(x_2^*)
=\left\{\gamma x_1\mid \gamma\ge -1+ \sqrt 3/2\right\}.\label{E5.4-8}
\end{align}	
By Theorem~\ref{T2.7}\eqref{T2.7.2}, 
\begin{gather*} 
\text{\rm Sol}(P)=\text{\rm \textbf{A}}(\Omega_1,x_1^*)\cap
\text{\rm \textbf{A}}(\Omega_2,x_2^*)=\left\{\gamma x_1\mid-1+\sqrt 3/2 \le\gamma\le1-\sqrt 3/2\right\}.
\end{gather*}

\if{
By Lemma~\ref{L1.1}\eqref{L1.1-1},  there exist unique vectors $\bar x_1, \bar x_2\in X$ such that
$\|\bar x_i\|=1$, 
$\langle x_i^*,u\rangle =\langle \bar x_i,u\rangle$ for all $u\in X$ $(i=1,2)$, and
\begin{align*}
\text{\rm\textbf{A}}_{\|\cdot\|}(\Omega_1,x_1^*)&=\left\{\omega_1+\gamma\cdot\bar x_1\mid \omega_1\in \Omega_1,\;\gamma\ge 0,\;\langle \bar x_1,w_1-\omega_1\rangle\le 0\;\;\text{for all}\;\;w_1\in\Omega_1\right\},\\
\text{\rm\textbf{A}}_{\|\cdot\|}(\Omega_2,x_2^*)&=\left\{\omega_2+\gamma\cdot\bar x_2\mid \omega_2\in \Omega_2,\;\gamma\ge 0,\;\langle \bar x_2,w_2-\omega_2\rangle\le 0\;\;\text{for all}\;\;w_2\in\Omega_2\right\}.
\end{align*}
By \eqref{E3.9-1} and \eqref{E5.4-3}, we have
$\langle x_i^*,\bx-x_i\rangle=\|\bx-x_i\|$ $(i=1,2)$. 
Hence, $\langle \bar x_i,\bx-x_i\rangle=\|\bx-x_i\|$ $(i=1,2)$.
Since $\|\bar x_1\|=\|\bar x_2\|=1$, there exist scalars $\gamma_1, \gamma_2\ge 0$ such that 
$\bar x_1=\gamma_1(\bx-x_1)$ and $\bar x_2=\gamma_2(\bx-x_2)$. 
Using again the fact that $\|\bar x_1\|=\|\bar x_2\|=1$, we obtain $\gamma_1=\gamma_2=\sqrt{3}$. Then $\bar x_1=-\sqrt{3}x_1$, $\bar x_2=\sqrt{3}x_1$, and
\begin{gather}
\text{\rm \textbf{A}}_{\|\cdot\|}(\Omega_1,x_1^*)
=\left\{\omega_1-\gamma\cdot x_1\mid \omega_1\in \Omega_1,\;\gamma\ge 0,\;\langle x_1,w_1-\omega_1\rangle\ge 0\;\;\text{for all}\;\;w_1\in\Omega_1\right\},\label{E5.4-4}\\
\text{\rm \textbf{A}}_{\|\cdot\|}(\Omega_2,x_2^*)
=\left\{\omega_2+\gamma\cdot x_1\mid\omega_2\in \Omega_2,\;\gamma\ge 0,\;\langle x_1,w_2-\omega_2\rangle\le 0\;\;\text{for all}\;\;w_2\in\Omega_2\right\}.\label{E5.4-5}
\end{gather}
Define $M_1:=\left\{\omega_1\in\Omega_1\mid\langle x_1,w_1-\omega_1\rangle\ge 0\;\;\text{for all}\;\;w_1\in\Omega_1\right\}$.
Let $\omega_1\in M_1$. Observe that $\omega_1$ is a minimiser of the function $w_1\mapsto\langle x_1,w_1\rangle$ on $\Omega_1$. We have
\begin{align*}
\langle x_1,w_1\rangle=\langle x_1,x_1\rangle+\langle x_1,w_1-x_1\rangle\ge\frac{1}{3}-\|x_1\|\cdot\|w_1-x_1\|\ge\frac{1}{3}-\frac{1}{2\sqrt{3}},
\end{align*}
for all $w_1\in\Omega_1$. 
Then 
$M_1=\left\{\left(1-\frac{\sqrt{3}}{2}\right)x_1\right\}$. By \eqref{E5.4-4}, 
$\text{\rm \textbf{A}}_{\|\cdot\|}(\Omega_1,x_1^*)
=\left\{\gamma\cdot x_1\mid \gamma\le 1-\frac{\sqrt{3}}{2}\right\}$.
By a similar argument, we obtain
$\text{\rm \textbf{A}}_{\|\cdot\|}(\Omega_2,x_2^*)
=\left\{\gamma\cdot x_1\mid \gamma\ge -1+\frac{\sqrt{3}}{2}\right\}$.
It follows from \eqref{T2.7-2} that
$\text{\rm Sol}(P)=\text{\rm \textbf{A}}_{\|\cdot\|}(\Omega_1,x_1^*)\cap
\text{\rm \textbf{A}}_{\|\cdot\|}(\Omega_2,x_2^*)=\left\{\gamma\cdot x_1\mid-1+\frac{\sqrt{3}}{2}\le\gamma\le1-\frac{\sqrt{3}}{2}\right\}$.
}\fi
\end{example}	

\section{Generalised Chebyshev Centre Problem}\label{S5}
We study a particular case of problem \eqref{P} with $\phi:=\|\cdot\|_{\psi_\infty}$.
The objective function is of the form:
\begin{gather}\label{5.1}
f(x)=\max\{d(x,\Omega_1),\ldots,d(x,\Omega_n)\}\;\;\text{for all}\;\;x\in X.
\end{gather}	
\if{
Define the active index set:
\begin{gather*}
	I(u):=\Big\{i\in\{1,\ldots,n\}\mid \|u-v_i\|=\max_{1\le j\le n}\|u-v_j\|\Big\}\;\;\text{for all}\;\;u\in X.
\end{gather*}	
}\fi

The following statement provides optimality conditions and a formula for the solution set of the GCCP.
\begin{theorem}\label{T5.1}
Let $\bar{x}\notin\bigcup_{i=1}^n\Omega_i$.
The following assertions hold.
\begin{enumerate}
\item\label{T5.1-1}
$\bar x\in \text{\rm Sol}(P,\|\cdot\|_{\psi_\infty})$ if and only if
there exist $(\lambda_1,\ldots,\lambda_n)\in\R^n_+$ and $x_i^*\in \partial d(\cdot,\Omega_i)(\bx)$ $(i=1,\ldots,n)$ such that
\begin{gather}\label{T5.1-6}
\sum_{i=1}^{n}\lambda_ix^*_i=0,\; 
\sum_{i=1}^{n}\lambda_i=1,\;
\left(d(\bx,\Omega_i)-\max_{1\le j\le n}d(\bx,\Omega_j)\right)\cdot\lambda_i=0\;\;(i=1,\ldots,n).
\end{gather}	
\item\label{T5.1-2}
Let $(\lambda_1,\ldots,\lambda_n)\in\R^n_+$ and $x_i^*\in \partial d(\cdot,\Omega_i)(\bx)$ $(i=1,\ldots,n)$ satisfy \eqref{T5.1-6}.
Then
\begin{gather*}
\text{\rm Sol}(P,\|\cdot\|_{\psi_\infty})=\bigcap_{\lambda_i>0}\text{\rm\textbf{B}}(\Omega_i,x_i^*),
\end{gather*}
where $\text{\rm \textbf{B}}(\Omega_i,x_i^*):=\text{\rm \textbf{A}}(\Omega_i,x_i^*)\cap \text{\rm\textbf{C}}(\Omega_i)$
with $\text{\rm \textbf{A}}(\Omega_i,x_i^*)$
given by \eqref{A+} and
\begin{gather}\label{T5.1-3}
\text{\rm\textbf{C}}(\Omega_i):=\left\{x\in X\mid d(x,\Omega_i)=\max_{1\le j\le n}d(x,\Omega_j)\right\}.
\end{gather}
\end{enumerate}
\end{theorem}

\begin{proof}
By Theorem~\ref{T3.5}\eqref{T3.5-1}, $\bar x\in \text{\rm Sol}(P,\|\cdot\|_{\psi_\infty})$ if and only if
there exist $(\lambda_1,\ldots,\lambda_n)\in\R^n_+$ and $x_i^*\in  \partial d(\cdot,\Omega_i)(\bx)$ $(i=1,\ldots,n)$ such that
\begin{gather}\label{T5.1-10}
\sum_{i=1}^{n}\lambda_ix^*_i=0,\; \sum_{i=1}^{n}\lambda_i=1 ,\;
\sum_{i=1}^{n}\lambda_id(\bar{x},\Omega_i)=\max_{1\le i\le n}d(\bar{x},\Omega_i).
\end{gather}	
Since $\sum_{i=1}^{n}\lambda_i=1$, the last condition in \eqref{T5.1-10} is equivalent to the last condition in \eqref{T5.1-6}.
This proves \eqref{T5.1-1}.	
By Theorem~\ref{T3.5}\eqref{T3.5-2}, a point $x\in\text{\rm Sol}(P,\|\cdot\|_{\psi_\infty})$ if and only if
\begin{gather}\label{T5.1-4}
x_i^*\in \partial d(\cdot,\Omega_i)(x)\;\;\text{with}\;\;\lambda_i>0\;\;\text{and}\;\;
\sum_{\lambda_i>0}\lambda_id(x,\Omega_i)=\max_{1\le j\le n}d(x,\Omega_j).
\end{gather}	
Since $\sum_{\lambda_i>0}\lambda_i=1$, the equality in \eqref{T5.1-4} is equivalent to $d(x,\Omega_i)=\max_{1\le j\le n}d(x,\Omega_j)$ with $\lambda_i>0$.
This proves \eqref{T5.1-2}.
\end{proof}

\if{
\begin{remark}\label{R5.2}
To the best of our knowledge, although the dual characterisations of the Chebyshev centre in Theorem~\ref{T5.1} follow directly from standard tools of convex and functional analysis, they have not been explicitly formulated in the literature.
In particular, the explicit description of the solution set in \eqref{T2.9-3} provides a  way to obtain all solutions from any given one together with its associated multipliers and dual vectors.
\end{remark}	
}\fi

\begin{remark}
\begin{enumerate}
\item 
An implicit form of  Theorem~\ref{T5.1}\eqref{T5.1-1} can be found in \cite[Proposition~3.10]{NamNguSal12}.
\item 
When $\Omega_1,\ldots,\Omega_n$  are singletons, Theorem~\ref{T5.1} reduces to \cite[Theorem~6.1]{Cuo26b}. 
In this case, the set \eqref{T5.1-3} coincides with the farthest Voronoi cell associated with the corresponding point; see \cite{GobMarTod19}. 
Algorithms for computing \eqref{T5.1-3} in the case of polygons and algebraic sets can be found  in \cite{Che11,Rie03}.
\end{enumerate}
\end{remark}

\begin{proposition}\label{P5.3}
Let  $\Omega_1,\ldots,\Omega_n\subset X$ be nonempty,
and  $\text{\rm\textbf{C}}(\Omega_i)$ $(i=1,\ldots,n)$ be given by \eqref{T5.1-3}.
The following assertions hold:
\begin{enumerate}
\item\label{P5.3-1}
$\text{\rm\textbf{C}}(\Omega_i)$ $(i=1,\ldots,n)$ are closed, and $\bigcup_{i=1}^{n}\text{\rm\textbf{C}}(\Omega_i)=X$;
\item\label{P5.3-2}
if $X$ is a Hilbert space, and $\Omega_1,\ldots,\Omega_n$ are closed and convex, then
\begin{gather*}
\text{\rm\textbf{C}}(\Omega_i)
=\bigcap_{j=1}^{n}\left\{x\in X\mid\ang{x, P_{\Omega_i}(x)-P_{\Omega_j}(x)} \le\frac{\|P_{\Omega_i}(x)\|^2-\|P_{\Omega_j}(x)\|^2}{2}\right\}
\end{gather*}
for all $i=1,\ldots,n$.
\end{enumerate}
\end{proposition}

\begin{proof}
\begin{enumerate}
\item 
Fix $i\in\{1,\ldots,n\}$ and define
\begin{gather*}
\varphi_i(x):=d(x,\Omega_i)-\max_{1\le j\le n}d(x,\Omega_j)\;\;\text{for all}\;\;x\in X.
\end{gather*}	
Then $\varphi_i$ is continuous and $\text{\rm\textbf{C}}(\Omega_i)=\varphi_i^{-1}(0)$.
Thus, $\text{\rm\textbf{C}}(\Omega_i)$ is closed. 
For every $x\in X$, we have $\varphi_{i_0}(x)=0$ for some $i_0\in\{1,\ldots,n\}$, and consequently,  $x\in\text{\rm\textbf{C}}(\Omega_{i_0})$. 
This proves the equality.
\item 
Let $i\in\{1,\ldots,n\}$. By \eqref{T5.1-3} and Lemma~\ref{L2.1}\eqref{L2.1-1},
\begin{align*}
\text{\rm\textbf{C}}(\Omega_i)
&=\bigcap_{j=1}^{n}\left\{x\in X\mid d(x,\Omega_i)\ge d(x,\Omega_j)\right\}\\
&=\bigcap_{j=1}^{n}\left\{x\in X\mid \|x-P_{\Omega_i}(x)\|\ge \|x-P_{\Omega_j}(x)\|\right\}\\
&=\bigcap_{j=1}^{n}\left\{x\in X\mid -2\ang{x,P_{\Omega_i}(x)}+\|P_{\Omega_i}(x)\|^2\ge -2\ang{x,P_{\Omega_j}(x)}+\|P_{\Omega_j}(x)\|^2\right\}\\
&=\bigcap_{j=1}^{n}\left\{x\in X\mid\ang{x, P_{\Omega_i}(x)-P_{\Omega_j}(x)} \le\frac{\|P_{\Omega_i}(x)\|^2-\|P_{\Omega_j}(x)\|^2}{2}\right\}.
\end{align*}
\end{enumerate}
The proof is complete.
\end{proof}

\begin{remark}
When $\Omega_1,\ldots,\Omega_n$ are singletons, Proposition~\ref{P5.3}\eqref{P5.3-2} recaptures the corresponding result in \cite[p.~309]{GobMarTod19}; see also \cite[Remark~6.2(ii)]{Cuo26b}.
\end{remark}	

\begin{example}\label{E5.3}
Let $\R^2$ be equipped with a norm $\|\cdot\|$,  vectors $x_1, x_2$ and sets $\Omega_1, \Omega_2$ be given by \eqref{E4.4-1} and \eqref{E4.4.2}.
\begin{enumerate}
\item\label{E5.3-2}
Suppose that $\|\cdot\|:=\|\cdot\|_{\psi_\infty}$.
By \eqref{E3.7-6} and \eqref{T5.1-3}, 
\begin{align*}
\text{\rm\textbf{C}}(\Omega_1)
&=\left\{u\in\R^2\mid d(u,\Omega_1)\ge d(u,\Omega_2)\right\}\\
&=\left\{u\in\R^2\mid\max\{\|u-x_1\|-1,0\}\ge
\max\{\|u-x_2\|-1,0\}\right\},\\
\text{\rm\textbf{C}}(\Omega_2)
&=\left\{u\in\R^2\mid d(u,\Omega_2)\ge d(u,\Omega_1)\right\}\\
&=\left\{u\in\R^2\mid\max\{\|u-x_2\|-1,0\}\ge
\max\{\|u-x_1\|-1,0\}\right\}.
\end{align*}	
Using the fact that $\max\{a,c\}\ge\max\{b,c\}$ if and only if $a\ge b$ or $c\ge b$ for all $a, b, c\in\R$, we obtain
\begin{align*}
\text{\rm\textbf{C}}(\Omega_1)
&=\left\{u\in\R^2\mid\|u-x_1\|\ge\|u-x_2\|\right\}\cup
\left\{u\in\R^2\mid\|u-x_2\|\le 1\right\}\\
&=\left\{(u_1,u_2)\mid\max\{|u_1+2|,|u_2|\}\ge\max\{|u_1-2|,|u_2|\}\right\}\\
&=\left\{(u_1,u_2)\mid u_1\ge 0\right\}\cup \left\{(u_1,u_2)\mid |u_1-2|\le|u_2|\right\},
\end{align*}
and
\begin{align*}
\text{\rm\textbf{C}}(\Omega_2)
&=\left\{u\in\R^2\mid\|u-x_2\|\ge\|u-x_1\|\right\}\cup
\left\{u\in\R^2\mid\|u-x_1\|\le 1\right\}\\
&=\left\{(u_1,u_2)\mid\max\{|u_1-2|,|u_2|\}\ge\max\{|u_1+2|,|u_2|\}\right\}\\
&=\left\{(u_1,u_2)\mid u_1\le 0\right\}\cup \left\{(u_1,u_2)\mid |u_1+2|\le|u_2|\right\}.
\end{align*}

\begin{figure}[H]
\centering
\begin{tikzpicture}[scale=0.9]
\begin{scope}
\draw[step=1cm,gray!50,dashed] (-3.5,-3.5) grid (3.5,3.5);

\foreach \x in {-3,-2,-1,0,1,2,3}{
\draw (\x,0.08)--(\x,-0.08) node[anchor=north,xshift=4pt] at (\x,0) {\small $\x$};}
\foreach \y in {-3,-2,-1,1,2,3}
\draw (0.08,\y)--(-0.08,\y) node[left] {\small $\y$};
\draw[-{Stealth}] (-3.5,0)--(3.5,0) node[below right] {$x$};
\draw[-{Stealth}] (0,-3.5)--(0,3.5) node[above left] {$y$};

\clip (-3.5,-3.5) rectangle (3.5,3.5);
\fill[red, fill opacity=0.5]
(-3.5,4)--(2,4)--(0,2)--(0,-2)--(2,-4)--(-3.5,-4);
\draw[red,  thick] (-3.5,4)--(2,4)--(0,2)--(0,-2)--(2,-4)--(-3.5,-4);
\node[red] at (-2,1.5) {$\text{\rm\textbf{C}}(\Omega_2)$};		
\fill[blue, fill opacity=0.5]
(3.5,4)--(-2,4)--(0,2)--(0,-2)--(-2,-4)--(3.5,-4);
\draw[blue,  thick] (3.5,4)--(-2,4)--(0,2)--(0,-2)--(-2,-4)--(3.5,-4);
\node[blue] at (2,1.5) {$\text{\rm\textbf{C}}(\Omega_1)$};		
\end{scope}
			
\begin{scope}[shift={(8,0)}]
\draw[step=1cm,gray!50,dashed] (-3.5,-3.5) grid (3.5,3.5);

\node[below right] at (3.5,0) {$x$};
\node[above left] at (0,3.5) {$y$};

\clip (-3.5,-3.5) rectangle (3.5,3.5);

\fill[red, fill opacity=0.5]
(-3.5,4)--(0,4)--(0,-4)--(-3.5,-4);
\draw[red,  thick] (-3.5,4)--(0,4)--(0,-4)--(-3.5,-4);
\node[red] at (-2,1.5) {$\text{\rm\textbf{C}}(\Omega_2)$};		
\fill[blue, fill opacity=0.5]
(3.5,4)--(0,4)--(0,-4)--(3.5,-4);
\draw[blue,  thick] (3.5,4)--(0,4)--(0,-4)--(3.5,-4);
\node[blue] at (2,1.5) {$\text{\rm\textbf{C}}(\Omega_1)$};	

\foreach \x in {-3,-2,-1,0,1,2,3}{
\draw (\x,0.08)--(\x,-0.08) node[anchor=north,xshift=4pt] at (\x,0) {\small $\x$};}
\foreach \y in {-3,-2,-1,1,2,3}
\draw (0.08,\y)--(-0.08,\y) node[left] {\small $\y$};
\draw[-{Stealth}] (-3.5,0)--(3.5,0);
\draw[-{Stealth}] (0,-3.5)--(0,3.5);
\end{scope}
\end{tikzpicture}
\caption{Example~\ref{E5.3}\eqref{E5.3-2} and \eqref{E5.3-1}}
\label{fig3}
\end{figure}
\item\label{E5.3-1}
Suppose that $\|\cdot\|:=\|\cdot\|_{\psi_p}$ with $p\in[1,\infty)$.
By the same argument as in assertion \eqref{E5.3-2}, 
\begin{align*}
\text{\rm\textbf{C}}(\Omega_1)
&=\left\{(u_1,u_2)\mid |u_1+2|\ge |u_1-2|\right\}
=\left\{(u_1,u_2)\mid u_1\ge 0\right\},\\
\text{\rm\textbf{C}}(\Omega_2)
&=\left\{(u_1,u_2)\mid |u_1-2|\ge |u_1+2|\right\}
=\left\{(u_1,u_2)\mid u_1\le 0\right\}.
\end{align*}
\end{enumerate}		
\end{example}	

The following examples illustrate  Theorem~\ref{T5.1}.
\begin{example}\label{E5.5}
Let $X:=\R^2$ be equipped with the norm $\|\cdot\|:=\|\cdot\|_{\psi_p}$ with $p\in[1,\infty]$, vectors $x_1, x_2, \bx$ and sets $\Omega_1, \Omega_2$ be given by \eqref{E4.4-1} and \eqref{E4.4.2}.
The objective function \eqref{5.1} is of the form
\begin{gather}\label{E5.5-0}
f(x)=\max\{d(x,\Omega_1),d(x,\Omega_2)\}
\end{gather}	
for all $x\in\R^2.$
By Example~\ref{E3.7}, we have $\bar x\in\text{\rm Sol}(P)$.
Define $x_1^*:=(1,0)$, $x_2^*:=(-1,0)$, and $\lambda_1=\lambda_2:=\frac{1}{2}$.
By Example~\ref{E4.4},  $x_1^*\in\partial d(
\cdot,\Omega_1)(\bx)$ and $x_2^*\in\partial d(
\cdot,\Omega_2)(\bx)$. 
By \eqref{E3.7-6},  $d(\bx,\Omega_1)=d(\bx,\Omega_2)$. 
Thus, condition \eqref{T5.1-6} is satisfied.
\begin{enumerate}
\item\label{E5.5-1}
Suppose that $\|\cdot\|:=\|\cdot\|_{\psi_\infty}$.
By \eqref{E4.4.4}, \eqref{E4.4.6},  and Example~\ref{E5.3}\eqref{E5.3-2},
\begin{align*}
\text{\rm \textbf{B}}(\Omega_1,x_1^*)
&=\text{\rm \textbf{A}}(\Omega_1,x_1^*)\cap \text{\rm\textbf{C}}(\Omega_1)
=\left\{(u_1,u_2)\mid u_1\ge 0,\;|u_2|\le 2+u_1\right\},\\
\text{\rm \textbf{B}}(\Omega_2,x_2^*)
&=\text{\rm \textbf{A}}(\Omega_2,x_2^*)\cap \text{\rm\textbf{C}}(\Omega_2)
=\left\{(u_1,u_2)\mid u_1\le 0,\;|u_2|\le 2-u_1\right\}.
\end{align*}	
By Theorem~\ref{T5.1}\eqref{T5.1-2}, 
$\text{\rm Sol}(P)=\text{\rm \textbf{B}}(\Omega_1,x_1^*)\cap \text{\rm \textbf{B}}(\Omega_2,x_2^*)=\left\{(0,u)\mid |u|\le 2\right\}.$
\item\label{E5.5-2}
Suppose that $\|\cdot\|:=\|\cdot\|_{\psi_p}$ with $p\in[1,\infty)$.
By \eqref{E4.4.8}, \eqref{E4.4.10},  and Example~\ref{E5.3}\eqref{E5.3-1},
\begin{align*}
\text{\rm \textbf{B}}(\Omega_1,x_1^*)
&=\text{\rm \textbf{A}}(\Omega_1,x_1^*)\cap \text{\rm\textbf{C}}(\Omega_1)
=\left\{(u,0)\mid u\ge 0\right\},\\
\text{\rm \textbf{B}}(\Omega_2,x_2^*)
&=\text{\rm \textbf{A}}(\Omega_2,x_2^*)\cap \text{\rm\textbf{C}}(\Omega_2)
=\left\{(u,0)\mid u\le 0\right\}.
\end{align*}	
By Theorem~\ref{T5.1}\eqref{T5.1-2}, $\text{\rm Sol}(P)=\text{\rm \textbf{B}}(\Omega_1,x_1^*)\cap \text{\rm \textbf{B}}(\Omega_2,x_2^*)=\{\bar x\}.$
\end{enumerate}
	
\begin{figure}[H]
\centering
\begin{tikzpicture}[scale=0.9]
\begin{scope}
\draw[step=1cm,gray!50,dashed] (-3.5,-3.5) grid (3.5,3.5);

\foreach \x in {-3,-2,-1,0,1,2,3}{
\draw (\x,0.08)--(\x,-0.08) node[anchor=north,xshift=4pt] at (\x,0) {\small $\x$};}
\foreach \y in {-3,-2,-1,1,2,3}
\draw (0.08,\y)--(-0.08,\y) node[left] {\small $\y$};
\draw[-{Stealth}] (-3.5,0)--(3.5,0) node[below right] {$x$};
\draw[-{Stealth}] (0,-3.5)--(0,3.5) node[above left] {$y$};

\clip (-3.5,-3.5) rectangle (3.5,3.5);			

\fill[red, fill opacity=0.5]
(-3.5,4)--(-2,4)--(0,2)--(0,-2)--(-2,-4)--(-3.5,-4);
\draw[red,  thick] (-3.5,4)--(-2,4)--(0,2)--(0,-2)--(-2,-4)--(-3.5,-4);
\node[red, font=\small] at (-2,1.5) {$\text{\rm \textbf{B}}(\Omega_2,x_2^*)$};					
\fill[blue, fill opacity=0.5]
(3.5,4)--(2,4)--(0,2)--(0,-2)--(2,-4)--(3.5,-4);
\draw[blue,  thick] (3.5,4)--(2,4)--(0,2)--(0,-2)--(2,-4)--(3.5,-4);
\node[blue] at (2,1.5) {$\text{\rm \textbf{B}}(\Omega_1,x_1^*)$};		
\end{scope}
			
\begin{scope}[shift={(8,0)}]
\draw[step=1cm,gray!50,dashed] (-3.5,-3.5) grid (3.5,3.5);	
\foreach \x in {-3,-2,-1,0,1,2,3}{
\draw (\x,0.08)--(\x,-0.08) node[anchor=north,xshift=4pt] at (\x,0) {\small $\x$};}
\foreach \y in {-3,-2,-1,1,2,3}
\draw (0.08,\y)--(-0.08,\y) node[left] {\small $\y$};
\draw[-{Stealth}] (0,-3.5)--(0,3.5) node[above left] {$y$};
\draw[-{Stealth}] (-3.5,0)--(3.5,0) node[below right] {$x$};

\draw[-{Stealth}, blue, very thick] (0,0)--(3.5,0);
\draw[red, very thick] (-3.5,0)--(0,0);
\node[red] at (-2,0.5) {$\text{\rm \textbf{B}}(\Omega_2,x_2^*)$};
\node[blue] at (2,0.5) {$\text{\rm \textbf{B}}(\Omega_1,x_1^*)$};
\fill(0,0) circle (2pt);
\node[above right] at (0,0) {$\bar x$};
\end{scope}
\end{tikzpicture}
\caption{Example~\ref{E5.5}\eqref{E5.5-1} and \eqref{E5.5-2}}
\end{figure}
\end{example}
\begin{example}
Let $X$ be the Hilbert space in Example~\ref{E4.3},
vectors $x_1, x_2, \bx$ and sets $\Omega_1, \Omega_2$ be given by \eqref{E5.4-1} and \eqref{E5.4-2}.
The objective function is of the form
\eqref{E5.5-0} for all $x\in X$.
By Example~\ref{E3.7}, $\bar x\in\text{\rm Sol}(P)$.
By Theorem~\ref{T5.1}, there exist $\lambda_1, \lambda_2\ge 0$ and $x_1^*\in \partial d(\cdot,\Omega_1)(\bx), x_2^*\in \partial d(\cdot,\Omega_2)(\bx)$  such that
\begin{gather}
\lambda_1x^*_1+\lambda_2x^*_2=0,\; 
\lambda_1+\lambda_2=1,\label{E5.6-1}\\
\left(d(\bx,\Omega_i)-\max\{d(\bx,\Omega_1),d(\bx,\Omega_2)\}\right)\cdot\lambda_i=0\;\;(i=1,2).
\label{E5.6-2}
\end{gather}	
By \eqref{sd}, $\|x_1^*\|^*=\|x_2^*\|^*=1$.
Combining this with \eqref{E5.6-1}, we have $\lambda_1, \lambda_2>0$. 
By \eqref{E3.7-6}, condition \eqref{E5.6-2} is satisfied. 
By \eqref{E3.7-6} and \eqref{T5.1-3}, 
\begin{align*}
\text{\rm \textbf{C}}(\Omega_1)
&=
\left\{x\in X\mid d(x,\Omega_1)\ge d(x,\Omega_2)
\right\}\\
&=
\left\{x\in X\mid\max\{\|x-x_1\|-1/2,0\}\ge \max\{\|x-x_2\|-1/2,0\}\right\}\\
&=\left\{x\in X\mid\|x-x_1\|\ge\|x-x_2\|\right\}\cup
\{x\in X\mid\|x-x_2\|\le 1/2\}\\
&=\left\{x\in X\mid\|x-x_1\|\ge\|x-x_2\|\right\}\\
&=\left\{x\in X\mid\langle x,x_1\rangle\le 0\right\}.
\end{align*}	
Similarly, $\text{\rm \textbf{C}}(\Omega_2)
=\left\{x\in X\mid\ang{x,x_1}\ge 0
\right\}$.
The sets $\text{\rm \textbf{A}}(\Omega_1,x_1^*)$ and  $\text{\rm \textbf{A}}(\Omega_2,x_2^*)$ are given by \eqref{E5.4-7} and \eqref{E5.4-8}. 
Then
\begin{align*}
\text{\rm \textbf{B}}(\Omega_1,x_1^*)
&=\text{\rm \textbf{A}}(\Omega_1,x_1^*)\cap \text{\rm\textbf{C}}(\Omega_1)
=\left\{\gamma x_1\mid\gamma\le 0\right\},\\
\text{\rm \textbf{B}}(\Omega_2,x_2^*)
&=\text{\rm \textbf{A}}(\Omega_2,x_2^*)\cap \text{\rm\textbf{C}}(\Omega_2)
=\left\{\gamma x_1\mid\gamma\ge 0\right\}.
\end{align*}	
By Theorem~\ref{T5.1}\eqref{T5.1-2}, 
$\text{\rm Sol}(P)=\text{\rm \textbf{B}}(\Omega_1,x_1^*)\cap\text{\rm \textbf{B}}(\Omega_2,x_2^*)=\{\bar x\}.$
\end{example}	

\section{Generalised $p$-Fermat-Torricelli Problem}\label{S6}
We study a particular case of problem \eqref{P} with $\phi:=\|\cdot\|_{\psi_p}$ for $p\in(1,\infty)$.
The objective function is of the form:
\begin{gather}\label{6.1}
f(x)=\left(d(x,\Omega_1)^p+\cdots+d(x,\Omega_n)^p\right)^{\frac{1}{p}}\;\;\text{for all}\;\;x\in X.
\end{gather}	

The following statement provides optimality conditions and a formula for the solution set of the  $p$-GFTP.
\begin{theorem}\label{T3.16}
Let  $\bar{x}\notin\bigcup_{i=1}^n\Omega_i$.
The following assertions hold.
\begin{enumerate}
\item\label{T3.16.1}
$\bar x\in \text{\rm Sol}(P,\|\cdot\|_{\psi_p})$ if and only if there exist $x_1^*,\ldots,x_n^*\in X^*$ such that
\begin{gather}\label{T3.16-2}
x_i^*\in \partial d(\cdot,\Omega_i)(\bx)\;\;(i=1,\ldots,n)\;\;\text{and}\;\; \sum_{i=1}^{n}d(\bx,\Omega_i)^{p-1}x^*_i=0.
\end{gather}
\item\label{T3.16.2}
Let  $x_1^*,\ldots,x_n^*\in X^*$ satisfy \eqref{T3.16-2}.
Then $\text{\rm Sol}(P,\|\cdot\|_{\psi_p})=\bigcap_{i=1}^{n}\text{\rm\textbf{D}}(\Omega_i,x_i^*),$
where $\text{\rm \textbf{D}}(\Omega_i,x_i^*):=\text{\rm \textbf{A}}(\Omega_i,x_i^*)\cap \text{\rm\textbf{E}}(\Omega_i,\bar x)$
with $\text{\rm \textbf{A}}(\Omega_i,x_i^*)$
given by \eqref{A+} and
\begin{gather}\label{T6.1-1}
\text{\rm\textbf{E}}(\Omega_i,\bar x):=\left\{x\in X\mid \frac{d(x,\Omega_i)^p}{\sum_{j=1}^{n}d(x,\Omega_j)^p}=\frac{d(\bx,\Omega_i)^p}{\sum_{j=1}^{n}d(\bx,\Omega_j)^p}\right\}.
\end{gather}
\end{enumerate}
\end{theorem}
	
\begin{proof}
By Theorem~\ref{T3.5}\eqref{T3.5-1},	$\bar x\in \text{\rm Sol}(P,\|\cdot\|_{\psi_p})$ if and only if there exist $(\lambda_1,\ldots,\lambda_n)\in\R^n_+$ and $x_i^*\in  \partial d(\cdot,\Omega_i)(\bx)$ $(i=1,\ldots,n)$ such that
\begin{gather}\label{T6.1-10}
\sum_{i=1}^{n}\lambda_ix^*_i=0,\;\;
\sum_{i=1}^{n}\lambda_i^q=1,\;\;
\sum_{i=1}^{n}\lambda_id(\bar{x},\Omega_i)=
\left(\sum_{i=1}^{n}d(\bx,\Omega_i)^p\right)^{1/p}.
\end{gather}
Since $\sum_{i=1}^{n}\lambda_i^q=1$, the last equality in \eqref{T6.1-10} is equivalent to
\begin{gather}\label{T6.1-3}
\lambda_i=\frac{d(\bx,\Omega_i)^{\frac{p}{q}}}{\left(\sum_{j=1}^{n}d(\bx,\Omega_j)^p\right)^\frac{1}{q}}>0
\quad (i=1,\ldots,n)
\end{gather}
with $q\in(1,\infty)$ satisfying $\frac{1}{p}+\frac{1}{q}=1$.
This proves \eqref{T3.16.1}.
By Theorem~\ref{T3.5}\eqref{T3.5-2}, a point $x\in\text{\rm Sol}(P,\|\cdot\|_{\psi_p})$ if and only if
\begin{gather}\label{T5.1-7}
x_i^*\in \partial d(\cdot,\Omega_i)(x)\;\;(i=1,\ldots,n)\;\;\text{and}\;\;
\sum_{i=1}^{n}\lambda_id(x,\Omega_i)=\left(\sum_{i=1}^{n}d(x,\Omega_i)^p\right)^{1/p}.
\end{gather}	
Since $\sum_{i=1}^{n}\lambda^q_i=1$, the equality in \eqref{T5.1-7} is equivalent to condition \eqref{T6.1-3} with $x$ in place of $\bx$. 
This proves \eqref{T3.16.2}.
\end{proof}
\begin{remark}
When $\Omega_1,\ldots,\Omega_n$  are singletons, Theorem~\ref{T5.1} reduces to \cite[Theorem~7.1]{Cuo26b}. 	
\end{remark}	

\if{
\begin{remark}
To the best of our knowledge, the dual necessary and sufficient optimality conditions and the construction of the solution set in Theorem~\ref{T3.16} have not been previously studied in the literature.
\end{remark}
}\fi
The following examples illustrate Theorem~\ref{T3.16}.
\begin{example}\label{E7.2}
Let $X:=\R^2$ be equipped with the norm $\|\cdot\|:=\|\cdot\|_{\psi_{p'}}$ with $p'\in[1,\infty]$.  
The dual norm is $\|\cdot\|^*=\|\cdot\|_{\psi_{q'}}$ with $q'\in[1,\infty]$ satisfying $\frac{1}{p'}+\frac{1}{q'}=1$.
Let vectors $x_1, x_2, \bx$ and sets $\Omega_1, \Omega_2$ be given by \eqref{E4.4-1} and \eqref{E4.4.2}.
The objective function \eqref{6.1} is of the form
\begin{gather}\label{E7.2-0}
f(x)=\left(d(x,\Omega_1)^p+d(x,\Omega_2)^p\right)^{\frac{1}{p}}
\end{gather}
for all $x\in\R^2$.
By Example~\ref{E3.7},  $\bar x\in\text{\rm Sol}(P)$.
Define $x_1^*:=(1,0)$ and $x_2^*:=(-1,0)$.
By Example~\ref{E4.4}, $x_1^*\in\partial d(
\cdot,\Omega_1)(\bx)$ and $x_2^*\in\partial d(
\cdot,\Omega_2)(\bx)$. 
By \eqref{E3.7-6}, $d(\bx,\Omega_1)=d(\bx,\Omega_2)$. 
Thus, condition \eqref{T3.16-2} is satisfied.
\begin{enumerate}
\item\label{E7.2-1}
Suppose that $\|\cdot\|:=\|\cdot\|_{\psi_\infty}$.
By \eqref{E3.7-6} and \eqref{T6.1-1},
\begin{align*}
\text{\rm \textbf{E}}(\Omega_1,\bar x)
&=\left\{u\in\R^2\mid
\frac{d(u,\Omega_1)^p}{d(u,\Omega_1)^p+d(u,\Omega_2)^p}=\frac{1}{2}\right\}\\
&=\left\{u\in\R^2\mid
d(u,\Omega_1)=d(u,\Omega_2)\right\}\\
&=\left\{u\in\R^2\mid\max\{\|u-x_1\|-1,0\}=
\max\{\|u-x_2\|-1,0\}\right\}\\
&=\left\{(u_1,u_2)\mid |u_2|\ge |u_1+2|,\;|u_2|\ge |u_1-2| \right\}
\cup\left\{(0,u_2)\mid |u_2|\le 2\right\},\\
\text{\rm \textbf{E}}(\Omega_2,\bar x)
&=\left\{u\in\R^2\mid \frac{d(u,\Omega_2)^p}{d(u,\Omega_1)^p+d(u,\Omega_2)^p}=\frac{1}{2}\right\}=\text{\rm \textbf{E}}(\Omega_1,\bar x).
\end{align*}	
The sets $\text{\rm \textbf{A}}(\Omega_1,x_1^*)$ and $\text{\rm \textbf{A}}(\Omega_2,x_2^*)$
are given by \eqref{E4.4.4} and \eqref{E4.4.6}. 
Then
\begin{align*}
\text{\rm \textbf{D}}(\Omega_1,x_1^*)
&=\left\{(u_1,u_2)\mid u_1\ge 0,\; |u_2|=2+u_1\right\}\cup\left\{(0,u_2)\mid |u_2|\le 2\right\},\\
\text{\rm \textbf{D}}(\Omega_2,x_2^*)
&=\left\{(u_1,u_2)\mid u_1\le 0,\; |u_2|=2-u_1\right\}\cup\left\{(0,u_2)\mid |u_2|\le 2\right\}.
\end{align*}	
By Theorem~\ref{T3.16}\eqref{T3.16.1}, 
$\text{\rm Sol}(P)=\text{\rm \textbf{D}}(\Omega_1,x_1^*)\cap \text{\rm \textbf{D}}(\Omega_2,x_2^*)=\left\{(0,u)\mid |u|\le 2\right\}$.

\begin{figure}[H]
\begin{tikzpicture}[x=1cm, y=0.5cm, scale=0.9]
\begin{scope}
\draw[step=2cm,gray!50,dashed] (-3.5,-5.5) grid (3.5,5.5);

\foreach \x in {-2,0,2}{
\draw (\x,0.08)--(\x,-0.08) node[anchor=north,xshift=4pt] at (\x,0) {\small $\x$};}
\foreach \y in {-2,2}
\draw (0.08,\y)--(-0.08,\y) node[left] {\small $\y$};
\draw[-{Stealth}] (-3.5,0)--(3.5,0) node[below right] {$x$};
\draw[-{Stealth}] (0,-5.5)--(0,5.5) node[above left] {$y$};

\clip (-3.5,-5.5) rectangle (3.5,5.5);
						
\fill[blue, fill opacity=0.5]
(-4,6)--(0,2)--(4,6);
\draw[blue,  thick] (-4,6)--(0,2)--(4,6);
\fill[blue, fill opacity=0.5]
(-4,-6)--(0,-2)--(4,-6);
\draw[blue,  thick] (-4,-6)--(0,-2)--(4,-6);
\draw[blue,  thick] (0,2)--(0,-2);
\node[blue] at (0.2,4.5) {$\text{\rm\textbf{E}}(\Omega_1,\bar x)=\text{\rm\textbf{E}}(\Omega_2,\bar x)$};
\node[blue] at (0.2,-4.5) {$\text{\rm\textbf{E}}(\Omega_1,\bar x)=\text{\rm\textbf{E}}(\Omega_2,\bar x)$};
\end{scope}
		
\begin{scope}[shift={(8,0)}]
\draw[step=2cm,gray!50,dashed] (-3.5,-5.5) grid (3.5,5.5);

\foreach \x in {-2,0,2}{
\draw (\x,0.08)--(\x,-0.08) node[anchor=north,xshift=4pt] at (\x,0) {\small $\x$};}
\foreach \y in {-2,2}
\draw (0.08,\y)--(-0.08,\y) node[left] {\small $\y$};
\draw[-{Stealth}] (-3.5,0)--(3.5,0) node[below right] {$x$};
\draw[-{Stealth}] (0,-5.5)--(0,5.5) node[above left] {$y$};

\clip (-3.5,-5.5) rectangle (3.5,5.5);
								
\draw[red,  thick] (-4,6)--(0,2)--(0,-2)--(-4,-6);
\node[red] at (-2,2.5) {$\text{\rm \textbf{D}}(\Omega_2,x_2^*)$};

\draw[blue, thick] (4,6)--(0,2)--(0,-2)--(4,-6);
\node[blue] at (2,2.5) {$\text{\rm \textbf{D}}(\Omega_1,x_1^*)$};
\end{scope}
\end{tikzpicture}
\caption{Example~\ref{E7.2}\eqref{E7.2-1}}
\label{fig5}
\end{figure}
\item\label{E7.2-2}
Suppose that $\|\cdot\|:=\|\cdot\|_{\psi_{p'}}$ with $p'\in[1,\infty)$.
By \eqref{E3.7-6} and \eqref{T6.1-1},
\begin{align*}
\text{\rm \textbf{E}}(\Omega_1,\bar x)
&=\left\{u\in\R^2\mid
\frac{d(u,\Omega_1)^p}{d(u,\Omega_1)^p+d(u,\Omega_2)^p}=\frac{1}{2}\right\}\\
&=\left\{u\in\R^2\mid
d(u,\Omega_1)=d(u,\Omega_2)\right\}
=\left\{(0,u_2)\mid u_2\in\R\right\},
\\
\text{\rm \textbf{E}}(\Omega_2,\bar x)
&=\left\{u\in\R^2\mid \frac{d(u,\Omega_2)^p}{d(u,\Omega_1)^p+d(u,\Omega_2)^p}=\frac{1}{2}\right\}=\text{\rm \textbf{E}}(\Omega_1,\bar x).
\end{align*}	
The sets $\text{\rm \textbf{A}}(\Omega_1,x_1^*)$ and $\text{\rm \textbf{A}}(\Omega_2,x_2^*)$
are given by \eqref{E4.4.8} and \eqref{E4.4.10}. 
Then
\begin{align*}
\text{\rm \textbf{D}}(\Omega_1,x_1^*)
=\text{\rm \textbf{A}}(\Omega_1,x_1^*)\cap \text{\rm\textbf{E}}(\Omega_1,\bar x)=\{\bx\},\;\;
\text{\rm \textbf{D}}(\Omega_2,x_2^*)
=\text{\rm \textbf{A}}(\Omega_2,x_2^*)\cap \text{\rm\textbf{E}}(\Omega_2,\bar x)=\{\bx\}.
\end{align*}	
By Theorem~\ref{T3.16}\eqref{T3.16.2},  
$\text{\rm Sol}(P)=\text{\rm \textbf{D}}(\Omega_1,x_1^*)\cap \text{\rm \textbf{D}}(\Omega_2,x_2^*)=\{\bx\}$.
\begin{figure}[H]
\centering
\begin{tikzpicture}[x=1cm, y=0.5cm, scale=0.9]
\begin{scope}
\draw[step=2cm,gray!50,dashed] (-2.5,-5.5) grid (3.5,5.5);
\foreach \x in {-2,-1,0,2}{
\draw (\x,0.08)--(\x,-0.08) node[anchor=north,xshift=4pt] at (\x,0) {\small $\x$};}
\draw[-{Stealth}] (-2.5,0)--(3.5,0) node[below right] {$x$};
\draw[-{Stealth}] (0,-5.5)--(0,5.5) node[above left] {$y$};

\draw[-{Stealth}, blue, very thick] (-1,0)--(3.5,0);
\draw[-{Stealth}, red, very thick] (0,-5.5)--(0,5.5);
					
\node[blue] at (2,1) {$\text{\rm \textbf{A}}(\Omega_1,x_1^*)$};
\node[red] at (1,-3) {$\text{\rm \textbf{E}}(\Omega_1,\bar x)$};
\fill(0,0) circle (2pt);
\fill[blue](-1,0) circle (2pt);
\node[above right] at (0,0) {$\bar x$};
\end{scope}
				
\begin{scope}[shift={(8,0)}]
\draw[step=2cm,gray!50,dashed] (-3.5,-5.5) grid (2.5,5.5);
\foreach \x in {-2,0,1,2}{
\draw (\x,0.08)--(\x,-0.08) node[anchor=north,xshift=4pt] at (\x,0) {\small $\x$};}
\draw[-{Stealth}] (-3.5,0)--(2.5,0) node[below right] {$x$};
\draw[-{Stealth}] (0,-5.5)--(0,5.5) node[above left] {$y$};
					
\draw[blue, very thick] (-3.5,0)--(1,0);
\draw[-{Stealth}, red, very thick] (0,-5.5)--(0,5.5);
					
\node[blue] at (-2,1) {$\text{\rm \textbf{A}}(\Omega_2,x_2^*)$};
\node[red] at (1,-3) {$\text{\rm \textbf{E}}(\Omega_2,\bar x)$};
\fill(0,0) circle (2pt);
\fill[blue](1,0) circle (2pt);
\node[above right] at (0,0) {$\bar x$};
\end{scope}
\end{tikzpicture}
\caption{Example~\ref{E7.2}\eqref{E7.2-2}}
\label{fig6}
\end{figure}
\end{enumerate}
\end{example}

\begin{example}
Let $X$ be the Hilbert space in Example~\ref{E4.3},
vectors $x_1, x_2, \bx$ and sets $\Omega_1, \Omega_2$ be given by \eqref{E5.4-1} and \eqref{E5.4-2}.
The objective function is of the form \eqref{E7.2-0} for all $x\in X$.
By Example~\ref{E3.7},  $\bar x\in\text{\rm Sol}(P)$.
Define $x_1^*:=-\sqrt{3}x_1$ and $x_2^*:=-\sqrt{3}x_2$. 
By Example~\ref{E4.3},  $x_1^*\in\partial d(
\cdot,\Omega_1)(\bx)$ and $x_2^*\in\partial d(
\cdot,\Omega_2)(\bx)$. 
By \eqref{E3.7-6},  $d(\bx,\Omega_1)=d(\bx,\Omega_2)$. 
Thus, condition \eqref{T3.16-2} is satisfied.
By \eqref{E3.7-6} and \eqref{T6.1-1},
\begin{align*}
\text{\rm \textbf{E}}(\Omega_1,\bar x)
&=\left\{x\in X\mid
\frac{d(x,\Omega_1)^p}{d(x,\Omega_1)^p+d(x,\Omega_2)^p}=\frac{1}{2}\right\}\\
&=\left\{x\in X\mid
d(x,\Omega_1)=d(x,\Omega_2)\right\}\\
&=\left\{x\in X\mid\langle x,x_1\rangle=0\right\},\\
\text{\rm \textbf{E}}(\Omega_2,\bar x)
&=\left\{x\in X\mid \frac{d(x,\Omega_2)^p}{d(x,\Omega_1)^p+d(x,\Omega_2)^p}=\frac{1}{2}\right\}
=\text{\rm \textbf{E}}(\Omega_1,\bar x).
\end{align*}	
The sets $\text{\rm \textbf{A}}(\Omega_1,x_1^*)$ and  $\text{\rm \textbf{A}}(\Omega_2,x_2^*)$ are given by \eqref{E5.4-7} and \eqref{E5.4-8}. 
Then
\begin{align*}
\text{\rm \textbf{D}}(\Omega_1,x_1^*)=\text{\rm \textbf{A}}(\Omega_1,x_1^*)\cap \text{\rm\textbf{E}}(\Omega_1,\bar x)=\{\bx\},\;\;
\text{\rm \textbf{D}}(\Omega_2,x_2^*)=\text{\rm \textbf{A}}(\Omega_2,x_2^*)\cap \text{\rm\textbf{E}}(\Omega_2,\bar x)=\{\bx\}.
\end{align*}	
By Theorem~\ref{T3.16}\eqref{T3.16.2}, 
$\text{\rm Sol}(P)=\text{\rm \textbf{D}}(\Omega_1,x_1^*)\cap \text{\rm \textbf{D}}(\Omega_2,x_2^*)=\{\bar x\}$.
\end{example}	

\section{Conclusions}\label{S7}
A unified framework for norm minimisation problems involving distances to convex sets has been developed. Complete dual necessary and sufficient optimality conditions and solution set descriptions have been studied. 
The latter are constructed from the dual vectors arising from the optimality conditions at a given solution. 
As a consequence, these results are applied to three important problems: the GFTP, GCCP and $p$-GFTP. 
Several examples in finite and infinite dimensional spaces are presented to illustrate the results.
In particular, it is  demonstrated that the solution set depends strongly on the choice of norm on the underlying space.

Potential directions for future research include extending the current framework to nonconvex settings, developing numerical algorithms based on the derived dual optimality conditions, and studying the stability and sensitivity of solution sets under perturbations of the problem data.

\section*{Acknowledgments}
The authors wish to thank Professor Alexander Kruger for  his comments  which helped us improve the manuscript.
\section*{Disclosure statement}
The authors report there are no competing interests to declare.

\section*{Data availability statement}
Data sharing is not applicable to this article as no new data were created or analysed in this study.




\end{document}

%% file: macro.tex
\newcommand{\al}{\alpha}
\newcommand{\be}{\beta}

\newcommand{\bx}{\bar x}

\newcommand {\R} {\mathbb R}

\newcommand {\B} {\mathbb B}

\newcommand {\dom} {{\rm dom}\,}

\newcommand {\cl} {{\rm cl}\,}
\newcommand {\co} {{\rm co}\,}

\newcommand {\sd} {\partial}

\newcommand{\ds}{\displaystyle}
\newcommand{\vertiii}[1]{\left\vert\kern-0.25ex\left\vert\kern-0.25ex\left\vert #1\right\vert\kern-0.25ex\right\vert\kern-0.25ex\right\vert}
\newcommand{\vertiiiBig}[1]{\Big\vert\kern-0.25ex\Big\vert\kern-0.25ex\Big\vert #1\Big\vert\kern-0.25ex\Big\vert\kern-0.25ex\Big\vert}

\newcommand{\abs}[1]{\left\vert#1\right\vert}

\newcommand{\ang}[1]{\left\langle #1 \right\rangle}

\newcounter{mycount}

%% file: Norm_Minimization_Problem-6+.bbl
\begin{thebibliography}{99}
\bibitem{NamHoa13} Nam NM, Hoang N. A generalized Sylvester problem and a generalized Fermat--Torricelli problem. J Convex Anal. 2013;20(3):669–687.

\bibitem{MorNam11}
Mordukhovich BS, Nam NM.
Applications of variational analysis to a generalized Fermat--Torricelli problem.
J Optim Theory Appl. 2011;148(3):431--454. 
doi: 10.1007/s10957-010-9761-7

\bibitem{AmiMac84} Amir D, Mach J. Chebyshev centers in normed spaces. J  Approx Theory. 1984;40(4):364--374.
doi: 10.1016/0021-9045(84)90011-X
	
\bibitem{NamHoaAn14}
Nam NM, Hoang N, An NT.
Constructions of solutions to generalized Sylvester and Fermat--Torricelli problems for Euclidean balls.
J Optim Theory Appl. 2014;160:483--509.
doi: 10.1007/s10957-013-0366-9
	
\bibitem{MarSwaWei02} Martini H, Swanepoel KJ, Weiss G. The Fermat–Torricelli problem in normed planes and spaces. J  Optim Theory Appl. 2002;115(2):283--314.
doi: 10.1023/A:1020884004689
	
\bibitem{KazLiuMor25} Kazemdehbashi S, Liu Y, Mordukhovich BS. New location science models with applications to UAV-based disaster relief. Preprint,  2025, arXiv:2510.15229.
	
\bibitem{BolMarSol99} Boltyanski V, Martini H, Soltan V. Geometric methods and optimization problems.
Springer Science \& Business Media; 1998.
	
\bibitem{BotTur94}
Botkin ND, Turova-Botkina VL.
An algorithm for finding the Chebyshev center of a convex polyhedron.
Appl Math Optim. 1994;29:211--222.
doi: 10.1007/BF01204183

\bibitem{NamNguSal12}
Nam NM, Nguyen TA, Salinas J.
Applications of convex analysis to the smallest intersecting ball problem.
J Convex Anal. 2012;19(2):497--518.

\bibitem{NamAnRecSun14}
Nam NM, An NT, Rector RB, Sun J. Nonsmooth algorithms and Nesterov's smoothing technique for generalized Fermat--Torricelli problems. SIAM J Optim. 2014;24(4):1815--1839.
doi: 10.1137/130945442

\bibitem{MorNamVil13}
Mordukhovich BS. Nam NM, Villalobos C. The smallest enclosing ball problem and the smallest intersecting ball problem: existence and uniqueness of solutions. Optim Lett. 2013;7(5):839--853.
doi: 10.1007/s11590-012-0483-7

\bibitem{AliTsa19} Alimov AR, Tsar’kov IG. Chebyshev centres, Jung constants, and their  applications. Russ Math Surv. 2019;74(5):775--849.
doi: 10.1070/RM9839

\bibitem{Man88}
Mangasarian OL. A simple characterization of solution sets of convex programs. Operations Research Letters. 1988;7(1):21--26.

\bibitem{Bur91}
Burke JV, Ferris MC. Characterization of solution sets of convex programs. Oper Res Lett. 1991;10(1):57--60.
doi: 10.1016/0167-6377(91)90087-6

\bibitem{Cuo26b}
Cuong ND,
Dual characterizations of norm minimization problems,
to appear in Optimization;
arXiv:2601.08153, 2026.

\bibitem{SaiKatTak00} Saito KS, Kato M, Takahashi Y. Absolute norms on $\mathbb{C}^n$. J Math Anal Appl. 2000;252(2):879--905. 
doi: 10.1006/jmaa.2000.7139

\bibitem{BonDun73}
Bonsall FF, Duncan J.
\emph{Numerical Ranges II}.
London Math Soc Lecture Note Ser, vol. 10.
Cambridge University Press, Cambridge; 1973.


\bibitem{Cuo26a}
Cuong ND,
Primal and dual characterizations of sign-symmetric norms,
Positivity 30 (2026), Article 35.
doi: 10.1007/s11117-026-01193-9

\bibitem{Mor06.1} Mordukhovich BS. {Variational analysis and generalized differentiation. I: Basic theory}. Berlin:
Springer; 2006. (Grundlehren der Mathematischen Wissenschaften [Fundamental principles
of mathematical sciences]; Vol. 330).

\bibitem{RocWet98} Rockafellar RT, Wets RJB. {Variational analysis}. Springer Berlin; 1998.

\bibitem{Bre11} Brezis H. {Functional analysis, Sobolev spaces and partial differential equations}. Springer New York; 2011.
	
\bibitem{Pen13} Penot JP. {Calculus without derivatives}. Graduate Texts in Mathematics, Vol. 266. Springer New York; 2013.
	
\bibitem{MorNam22} Mordukhovich BS, Nam NM. {Convex analysis and beyond. Volume I: Basic theory}. Springer Cham; 2022.
	
\bibitem{Cla90}
Clarke FH.
Optimization and Nonsmooth Analysis.
SIAM, Philadelphia; 1990.


\bibitem{GobMarTod19} Goberna MA, Martínez-Legaz JE, Todorov MI. On farthest Voronoi cells. Linear Algebra Appl. 2019;583:306--322.
doi: 10.1016/j.laa.2019.09.002

\bibitem{Che11}
Cheong O, Everett H, Glisse M, Gudmundsson J, Hornus S, Lazard S, Lee M, Na HS. Farthest-polygon Voronoi diagrams. Comput Geom. 2011;44(4):234--247.
doi: 10.1016/j.comgeo.2010.11.004

\bibitem{Rie03}
Rieger JH. Voronoi diagrams of real algebraic sets. Geom Dedicata. 2003;98(1):81--94.
doi: 10.1023/A:1023633014914


%
%
\end{thebibliography}
